\documentclass[a4,12pt]{amsart}
\usepackage{amssymb,amstext,amsmath,amscd,amsthm,amsfonts,enumerate,graphicx,latexsym, hyperref,cleveref,mathtools,mathrsfs,stmaryrd}
\usepackage[usenames]{color}
\usepackage[all]{xy}

 \usepackage{comment}
 \usepackage{mathrsfs}
\usepackage{tikz-cd}
\usepackage{tikz}
\usepackage{changepage}
\usetikzlibrary{shapes,arrows,fit,calc,positioning}
\usetikzlibrary{matrix}
\xyoption{all}
\usepackage{mathtools}
\newtheorem{thm}{Theorem}[section]

\newtheorem{cor}[thm]{Corollary}

\newtheorem{lem}[thm]{Lemma}
\newtheorem{prop}[thm]{Proposition}

\theoremstyle{definition}
\newtheorem{dfn}[thm]{Definition}
\newtheorem{rem}[thm]{Remark}

\newtheorem{ex}[thm]{Example}
\newtheorem{conv}[thm]{Convention}

\newtheorem*{claim*}{Claim}
\theoremstyle{remark}

\newcommand{\PN}{\Phi_n(x)}
\numberwithin{equation}{thm}

\def\ass{\operatorname{Ass}}

\def\Hom{\operatorname{Hom}}

\def\ker{\operatorname{Ker}}

\def\m{\mathfrak{m}}

\def\n{\mathfrak{n}}

\def\D{\mathfrak{D}}

\def\spec{\operatorname{Spec}}

\def\spec{\operatorname{Spec}}

\begin{document}
\setlength{\baselineskip}{15pt}
\title[On abelian extensions in mixed characteristic]{On abelian extensions in mixed characteristic and ramification in codimension one}
\author{Daniel Katz}
\address{Department of Mathematics, University of Kansas, Lawrence, KS 66045-7523, USA}
\email{dlk53@ku.edu}

\author{Prashanth Sridhar}
\address{Department of Mathematics and Statistics, Auburn University}
\email{prashanth.sridhar0218@gmail.com}
\begin{abstract} 
A theorem of Paul Roberts (\cite{RO}) states that the integral closure of a regular local ring in a generically abelian extension is Cohen-Macaulay, provided the characteristic of the residue field does not divide the order of the Galois group. An example of Koh in \cite{KO} shows the conclusion is false in the modular case. After a modification to the statement concerning ramification over $p$ in codimension one, we give an extension of Roberts's theorem to the modular case for unramified regular local rings in mixed characteristic when the $p$-torsion of the Galois group is annihilated by $p$. 
\end{abstract}

\subjclass[2010]{}
\maketitle


\section{Introduction} 

It is a classical question in commutative algebra and algebraic geometry to study the variety $\spec(R)$ in terms of the fibres of a finite morphism $f:\spec(R)\to \spec(S)$, where $R$ is a normal domain and $S$ is regular. Such an $S$ is available when $R$ is a finitely generated algebra over a field or is complete. For example, the purity of branch locus states that if $f$ is unramified in codimension one, then $f$ is \'{e}tale, see \cite{Zariski_pbl}, \cite{Nagata_remarks_pbl}, \cite{Nagata_pbl} and \cite{Auslander_pbl}. Generalizations and variants of this theorem have been studied extensively by relaxing the hypothesis that $S$ is regular and establishing whether good properties of $S$ transfer to good properties of $R$ when there is no ramification in low codimension, see for example \cite{SGA2}, \cite{Griffith1987}, \cite{Griffith1991}, \cite{Cutkosky95}, \cite{Kantorovitz1999}. One such property whose transfer has been studied is that of Cohen-Macaulayness. The work in this paper fits in the framework of a related, but orthogonal question - are there good patterns of ramification in codimension one that result in transfer of good properties (Cohen-Macaulayness) from (regular) $S$ to $R$? As far as we know, very little is known in this direction. 

\par 
Our study is motivated by a theorem of Roberts in \cite{RO} that states that the integral closure of a regular local ring in an abelian extension of its fraction field is Cohen-Macaulay, provided the characteristic of the residue field does not divide the order of the Galois group. This result has seen generalizations/applications to the theory of algebraic monoids, singularity theory and arithmetic schemes with a tame action over an abelian group, see \cite{Renner1983, Itoh1989, CPT2009}. We explain by means of an alternate proof of this statement (see \Cref{thm:Roberts}) as to why we view this result as one about ``good ramification" in codimension one. Roberts's theorem fails in the modular case, i.e., when the characteristic of the residue field divides the order of the Galois group. Koh in \cite{KO} gave an example of this phenomenon in mixed characteristic. One way to explain this is to note that Roberts's theorem relies on Maschke's theorem and there is no direct analog of such an argument in the modular case. But beyond this, not much seems to be known in this regard; see \cite{GRIFFITH2015502} for comments. Guided by intuition from our alternate proof of Roberts's theorem, we ask if there exists some analog of this theorem in mixed characteristic. 

\par Let $\Psi:S\to R$ be a map of commutative Noetherian rings and $p\in\mathbb{Z}$ a prime integer. Say $\Psi$ is \textit{$p$-unramified} if $\Psi$ is \'{e}tale in codimension one over $p$, i.e., if $S_P\to R_P$ is \'{e}tale for all height one primes $P\subseteq S$ containing $p$ (\Cref{def:p_unramified_morphism}). Similarly, in direct analogue to the notion in algebraic number theory, $\Psi$ is tamely $p$-ramified if it is so in codimension one over $p$, see \Cref{def:tame_p_ramification}. If $p\in S$ is a unit, these conditions are satisfied vacuously. It is reasonable to expect that in mixed characteristic $p>0$, if $S\to R$ is generically abelian with $S$ regular, $R$ normal and $S\to R$ $p$-unramified, then $R$ is Cohen-Macaulay. Unfortunately, this is not true either, as evidenced by Koh's example, see \Cref{ex:Koh}. Hence, to get an extension of Roberts's theorem to mixed characteristic, we turn to Kummer theory. Let $X$ be an indeterminate over $S$ and $p\in\mathbb{Z}$ a prime integer. Then we say $f\in S$ is \textit{$p$-unramified} (resp. \textit{tamely $p$-ramified}) if for some root $\omega$ of the polynomial $X^p-f\in S[X]$, $S\to \overline{S[\omega]}$ is $p$-unramified (resp. tamely $p$-ramified), where $\overline{*}$ denotes normalization (\Cref{def:p-unramified_element}).  Otherwise, we say $f\in S$ is $p$-ramified. The definitions are independent of the choice of the root $\omega$ if $S$ possesses a primitive $p$-th root of unity. We characterize these properties in terms of numerical conditions in codimension one using the function $\Gamma_I$ (see \Cref{conv1} for the $\Gamma_I$ notation):

\begin{thm}[\Cref{thm:p_ramification_characterization}]
   Let $S$ be a regular local ring such that $\mathrm{char}(\mathrm{Frac}(S))=0$. Assume $S$ possesses a primitive $p$-th root of unity for $p\in \mathbb{Z}$ a prime integer. Then the following are equivalent:
  \begin{enumerate}
      \item $0\neq f\in S$ is $p$-unramified.
      \item $0\neq f\in S$ is tamely $p$-ramified.

    \item either
    \begin{enumerate}
      \item $f\in S$ is a $p$-th power or
      \item $f\notin \bigcup_{Q\in \ass(S/(p))} Q$ and for all $Q\in \ass(S/(p))$

   \[\Gamma_{QS_Q}(f)\geq [\sum_{i=0}^{\infty}(1/p^i)]ord_Q(p)=\dfrac{p}{p-1}ord_Q(p)\]

  where $\Gamma_{QS_Q}(f)$ is the largest power $t$ of $Q$ such that $f$ admits a $p$-th root in $S_Q/Q^tS_Q$.
  \end{enumerate}
    
  \end{enumerate} 

\end{thm}

Assume $S$ has mixed characteristic $p>0$ and that it possesses a primitive $p$-th root of unity. Given a generically abelian extension $S\to R$, with $R$ normal, one has a canonical choice of elements and codimension one primes (which we call canonical divisors) in $S$ associated to it; this is explained in \Cref{sec:Abelian_extensions&p_unramified_property}. For instance, in the modular setup, if $S\to R$ is tamely $p$-ramified, then the canonical divisors are precisely the codimension one primes in $S$ away from $p$ that ramify in $R$, see \Cref{rem:canonical_divisors_modular_setup}. If in addition to $S\to R$ being tamely $p$-ramified, each canonical divisor in $S$ is tamely $p$-ramified (the analogous requirement is always satisfied in the non-modular setup), one obtains the following extension of Roberts's theorem for unramified regular local rings when the $p$-torsion of the Galois group is annihilated by $p$. In fact, our result is a bit more general (see \Cref{def:quasi_p_unramified} and \Cref{p-unramified_unramified_RLR} for the definition of abelian extensions of tamely $p$-ramified type):

\begin{thm}[\Cref{thm:roberts_extension}]
Let $S$ be an unramified regular local ring of mixed characteristic $p>0$ with quotient field $L$. Let $K/L$ be a finite abelian extension with $p$-torsion annihilated by $p$ and $R$ the integral closure of $S$ in $K$. If $K/L$ is of tamely $p$-ramified type over $S$, then $R$ admits a small Cohen-Macaulay algebra.
\end{thm}

We discuss how our results apply to Koh's example, see \Cref{ex:Koh}; in particular, Koh's example admits a small Cohen-Macaulay algebra. Finally, we observe in \Cref{cor:p-ramified} that the $p$-ramified canonical divisors are in some sense the obstruction to such an analog in full generality and present a calculation involving the first $p$-ramified case showing the existence of a small Cohen-Macaulay module of rank at most $(p-1)p^{p(d-1)+1}$, where $d=\dim(S)$.
\par

The paper is organized as follows. Section 2 contains preliminary definitions and results that are used later. Section 3 contains the alternate take on Roberts's theorem. Section 4 consists of the numerical characterization of the $p$-unramified property and the definitions of (quasi) $p$-unramified abelian extensions. Section 5 presents the main result and includes a discussion on Koh's example. Finally, section 6 comments on the $p$-ramified case.

\section{Preliminaries}

In this section, we present some definitions and prove some preliminary results in preparation for the sections that follow.

\begin{conv}\normalfont\label{conv1}
\noindent \begin{enumerate}
\item Rings are commutative and modules are finitely generated.
\item For an integer $n$, $\PN\in \mathbb{Z}[x]$ will denote the $n$-th cyclotomic polynomial.
\item A Noetherian ring $R$ admits a small Cohen-Macaulay (CM) algebra $T$ if there is an injective, module finite map of rings $R\rightarrow T$ such that $T$ is Cohen-Macaulay.

\item Suppose $S$ is a ring and $I\subsetneq S$ an ideal. For $m\in \mathbb{Z}$, $m\geq 0$, let $\phi_m:S\rightarrow S/I^m$ denote the natural map. Define
\[\Gamma_I: S \times \mathbb{Z}\rightarrow \mathbb{Z}_{\geq 0}\cup \{\infty\}\]
\[(f,n)\mapsto \{sup\{m\}\:|\:\sqrt[n]{\phi_m(f)}\in S/I^m\}.\]
Here $\sqrt[n]{\phi_m(f)}$ refers to any root of the polynomial $X^n-f\in S/I^m[X]$, where $X$ is an indeterminate over $S/I^m$.

\item A Noetherian ring $R$ of prime characteristic $p>0$ is \textit{$F$-finite} if the Frobenius endomorphism $F:R\to R$ makes $R$ into a module-finite $R$-algebra.
\end{enumerate}
\end{conv}

Recall the following:

\begin{dfn}[\cite{AB}]
 Let $S$ be a ring and $R$ an $S$-algebra. $P\in \spec(S)$ is unramified in $R$ if for all $Q\in \spec(R)$ lying over $P$, $PR_Q=QR_Q$ and $S_P/PS_P\rightarrow R_Q/QR_Q$ is a finite separable field extension. We say $R$ is unramified over $S$ if every $P\in \spec{S}$ is unramified in $R$. We say $R$ is \'{e}tale over $S$ if it is flat and unramified over $S$. 
\end{dfn}

\begin{rem}[\cite{AB}]\label{rem:normal_etale}
    Let $S\rightarrow R$ be a module finite extension of normal domains. Then $R$ is unramified over $S$ if and only if $R$ is \'{e}tale over $S$.
\end{rem}

 We say a ring map $S\to R$ is \'{e}tale over $a\in S$ in codimension one if $S_P\to R_P$ is \'{e}tale for all height one primes $P\subseteq S$ containing $a$.

 \begin{dfn}\label{def:p_unramified_morphism}
     Let $\Psi:S\to R$ be a map of Noetherian rings and $p\in\mathbb{Z}$ a prime integer. We say $\Psi$ is $p$-unramified if $\Psi$ is \'{e}tale in codimension one over $p$. Otherwise, we say $\Psi$ is $p$-ramified.
 \end{dfn}

 \begin{dfn}
 A local extension of DVRs $(V_1,\pi_1,k_1)\to(V_2,\pi_2,k_2)$ is \textit{tamely ramified} if the induced extension of residue fields is separable and $\mathrm{ord}_{\pi_2V_2}(\pi_1)$ is coprime to $\mathrm{char}(k_1)$. 
\end{dfn}

\begin{dfn}\label{def:tame_p_ramification}
 Let $\Psi:S\to R$ be a module finite map of normal domains and $p\in \mathbb{Z}$ a prime integer. We say $\Psi$ is \textit{tamely $p$-ramified}, if for all height one primes $Q\in \spec(R)$ containing $p$, $S_{Q\cap S}\to R_Q$ is tamely ramified.
\end{dfn}

\begin{dfn}\label{def:p-unramified_element}
  Let $S$ be a noetherian semi-local regular ring and $X$ an indeterminate over $S$. Let $p\in\mathbb{Z}$ be a prime integer and $n$ any integer. Then $f\in S$ is \textit{$p$-unramified over $n$} if for some root $\omega$ of the polynomial $X^n-f\in S[X]$, $S\to \overline{S[\omega]}$ is $p$-unramified, where $\overline{*}$ denotes normalization.  Otherwise, $f\in S$ is \textit{$p$-ramified over $n$}. If $n=p$, we just say $f\in S$ is $p$-unramified or $p$-ramified respectively. A subset $V\subseteq S$ is $p$-unramified over $n$ if each element of $V$ is so.
  \par
We say $f\in S$ is \textit{tamely $p$-ramified over $n$} if for some root $\omega$ of the polynomial $X^n-f\in S[X]$, $S\to \overline{S[\omega]}$ is tamely $p$-ramified. If $n=p$, we just say $f\in S$ is tamely ramified.

\end{dfn}

\begin{rem}
    In general, \Cref{def:p-unramified_element} is dependent on the choice of the $n$-th root. For instance, if $S$ is an unramified regular local ring of odd mixed characteristic $p$ and $f=h^p$ is a $p$-th power in $S$, then the extension corresponding to $\omega=h$ is \'{e}tale in codimension one over $p$, but the one corresponding to $\omega=h\epsilon$ for $\epsilon$ a primitive $n$-th root of unity is not. However, when $S$ possesses a primitive $n$-th root of unity, any two distinct $n$-th roots define the same extension and hence the definition is independent of the choice of $n$-th root. 
\end{rem}

\begin{rem}
    With notation as in \Cref{def:p-unramified_element}, if $q\mid n$ and $f\in S$ is $p$-unramified over $n$, then $f\in S$ is $p$-unramified over $q$. Similarly, if $f\in S$ is tamely $p$-ramified over $n$, then $f\in S$ is tamely $p$-ramified over $q$.
\end{rem}

We need \Cref{thm:p_ramification_characterization} to give interesting examples of \Cref{def:p-unramified_element}, but \Cref{ex:basic} lists a few basic ones. On the flip side, once we have \Cref{thm:p_ramification_characterization}, its explicit nature makes it easy to write down examples.

\begin{ex}\label{ex:basic}
\begin{enumerate}
    \item If $p\in S$ is a unit, then vacuously, every $f\in S$ is $p$-unramified.
    \item If $S$ is a regular local ring of mixed characteristic $p>0$, then $p\in S$ is $p$-ramified. More generally, any $f\in \bigcup_{P\in \ass(S/p)} P$ is $p$-ramified.
    \item If $S$ is an unramified regular local ring of mixed characteristic $p>0$ and $f\in S$ is not a $p$-th power modulo $pS$, then $f\in S$ is $p$-ramified. To see this, note that since $S/pS$ is integrally closed, $X^p-f$ is irreducible modulo $pS$, so that $p\in S[\omega]$ is prime and the induced extension of residue fields from $S_{(p)}\to S[\omega]_{(p)}$ is purely inseparable.
\end{enumerate}
      
\end{ex}

For a ring $S$, an element $x\in S$ is said to be square free if for all height one primes $x\in Q\subseteq S$, $QS_Q=xS_Q$. Throughout this paper, the notation $\sqrt[n]{a}$ for $a\in S$, refers to any root of the polynomial $X^n-a\in S[X]$ in some extension of the total quotient ring of $S$ (when there is no cause for confusion). 
We include the following two results from \cite{HK} for convenience:

\begin{prop}[\cite{HK}]\label{HK1}
Let $S$ be an integrally closed Noetherian domain and $n\in S$ a unit for some positive integer $n$. Let $a_1,\dots,a_r\in S$ be square free elements such that no two of them are contained in a single height one prime ideal. Then $a_2,\dots,a_r$ are square free in $S[\sqrt[n]{a_1}]$.
\end{prop}

\begin{prop}[\cite{HK}]\label{HK2}
Let $S$ be an integrally closed Noetherian domain and $n\in S$ a unit for some positive integer $n$. Let $a_1,\dots,a_r\in S$ be square free elements such that no two of them are contained in a single height one prime ideal. Then $R=S[\sqrt[n]{a_1},\dots,\sqrt[n]{a_r}]$ is integrally closed.
\end{prop}

We record a motley collection of observations that we will need later. The proofs rely mostly upon some standard facts, but we give details for the sake of completion.

\begin{prop}\label{prop:cyclotomic_irreducible}
    Let $S$ be an integrally closed domain such that $\mathrm{char}(\mathrm{Frac}(S))=0$ and $p \in S$ is a prime element for some prime integer $p$. Then $\Phi_{p^r}(x) \in \mathbb{Z}[x]$ is irreducible over $S$.
\end{prop}
\begin{proof}
     The proof is essentially the same as the case $S = \mathbb{Z}$, so we just provide a sketch. The point is that Eisenstein's criterion together with a change of variables still works in this setting. So, it suffices to show that $\Phi_{p^r}(x+1)$ is irreducible over $S$. To see this, recall that if $r > 1$ then $\Phi_{p^r}(x) = \Phi_p(x^{p^{r-1}})$, thus, 
\[
\Phi_{p^r}(x+1) = ((x+1)^{p^{r-1}})^{p-1}+ ((x+1)^{p^{r-1}})^{p-2}+ \cdots + (x+1)^{p^{r-1}} + 1.
\] For each $1\leq k\leq p-1$, $p^{r-1}\choose k$ is divisible by $p$. It follows that $\Phi_{p^r}(x+1)$ is an Eisenstein polynomial in $\mathbb{Z}[x]$. Since $p$ is prime in $S$, $\Phi_{p^r}(x+1)$ is Eisenstein in $S[x]$, and hence irreducible over $S$.
\end{proof}

\begin{lem}\label{lem:deg_extension}
  Let $A$ be an integrally closed Noetherian domain with quotient field $L$ and suppose $q_1, \ldots, q_s\in A$ satisfy the following: 
\begin{enumerate}
\item[(i)] $q_1, \ldots, q_l\in A$ are square-free non-units and $q_{l+1}, \ldots, q_s$ are units.
\item[(ii)] For $1\leq i\leq l$, no height one prime of $A$ contains two of the square free elements $q_1, \ldots, q_l$.
\end{enumerate}
 Suppose further that $A$ contains a primitive $n$-th root of unity and $n$ is a unit in $A$. Let $K$ denote the quotient field of $A[\sqrt[n]{q_1}, \ldots, \sqrt[n]{q_s}]$. Then the degree of $K$ over $L$ is $en^{l}$, where $e$ divides $n$. In particular,  the degree of $K$ over $L$ is a unit in $A$.
\end{lem} 

\begin{proof}
    Consider a single expression $\alpha := \sqrt[n]{q}$, an $n$-th root of $q$. If $q$ is a square-free non-unit, then $x^n-q$ is irreducible over $A$ (and $L$), by \cite{lang}, Theorem 9.1, and thus the degree of $K$ over $L$ equals $n$. Moreover, since $n$ is a unit in $A$, $A[\alpha]$ is integrally closed by \Cref{HK2}.  
Now suppose $q$ is a unit in $A$. Set $[K:L] = d$. We now note that $\alpha ^d \in A$, equivalently, $\alpha ^d \in L$. Let $\epsilon \in L$ denote a primitive $n$th root of unity and  $f(x)$ denote the minimal polynomial of $\alpha$ over $L$. On the one hand, since $\alpha$ is a root of $x^n-q$, $f(x)$ divides $x^n-q$ in $L[x]$. On the other hand, $x^n-q= 
(x-\alpha)(x-\alpha \epsilon)\cdots (x-\alpha\epsilon^{n-1})$ in $K[x]$. Thus $f(x) = (x-\alpha\epsilon^{i_1})\cdots (x-\alpha\epsilon^{i_d})$, for indices $i_1, \ldots, i_m$. Therefore, the constant term of $f(x)$, which belongs to $L$, is $\alpha^d\epsilon^t$, where $t = i_1+\cdots + i_d$, Since $\epsilon ^t \in L$, we have $\alpha ^d \in L$, as required. It follows that $d$ is the least positive integer with $\alpha ^d\in L$. Writing $n = dh+r$, with $0\leq r\leq d-1$, we have $\alpha ^n = \alpha ^{dh}\alpha ^r$, which implies $\alpha ^r\in L$. Thus, $r = 0 $ and $d$ divides $n$. Note that $A[\alpha]$ is also integrally closed in this case.
\par
For the general case, one proceeds by induction using the fact that $A[\sqrt[n]{q_1}]$ is a normal domain, and the hypotheses (i), (ii) are preserved in this ring (see \Cref{HK2}). 
\end{proof}

\begin{lem}\label{kernel_lemma}
Let $S$ be a domain and for each $1\leq i\leq n$, let $S\hookrightarrow R_i$ be module finite extension of domains such that $R_1\otimes_S\dots\otimes_SR_n$ is torsion free. Let $V$ denote the join of the $R_i$ in a fixed algebraic closure of $\mathrm{Frac}(S)$. Assume that $\deg(\prod_{j=1}^{n}\mathrm{Frac}(R_j):\mathrm{Frac}(S))=\prod_{j=1}^n\deg(\mathrm{Frac}(R_j):\mathrm{Frac}(S))$. Then $V\simeq R_1\otimes_S\dots\otimes_SR_n$ as $S$-algebras.
\end{lem}
\begin{proof}
Let $\Psi:R_1\otimes_S\dots\otimes_SR_n\rightarrow V$ denote the natural surjection of $S$-algebras. Set $L=Frac(S)$, $K=Frac(V)$ and $K_i:=Frac(R_i)$. By hypothesis, there exists an isomorphism of $L$-vector spaces $K_1\otimes_L\dots\otimes_LK_n\rightarrow K$. Thus  

\[id_{L}\otimes\Psi: L\otimes_S(R_1\otimes_S\dots\otimes_SR_n) \to L\otimes_SV\]
is a surjection of finite dimensional $L$-vector spaces of the same rank and hence an isomorphism. Since $L$ is $S$-flat, this implies $\ker(\Psi)$ is torsion. Since $R_1\otimes_S\dots\otimes_SR_n$ is torsion free, this implies $\Psi$ is injective and hence an isomorphism. 
\end{proof}

\begin{lem}\label{lem1}
Let $\psi:S\rightarrow R$ be a module finite homomorphism of rings. Suppose $R$ admits a finite module $M$ such that $M$ is $S$-free of rank $n$. Let $N$ be any $S$-module. Then $R$ admits a module $C$ such that $C\simeq N^{\oplus n}$ as $S$-modules.
\end{lem}
\begin{proof}
Note that $M$ defines a ring homomorphism $\phi:R\rightarrow \mathbb{M}_{n\times n}(S)$ such that $\phi(\psi(S))$ consists of scalar matrices. The map is injective if and only if $M$ is faithful over $R$. Set $C:=\mathbb{M}_{n\times 1}(N)\cong \Hom_S(S,N^{\oplus n})$. Then $C$ clearly admits an $R$-module structure via $\phi$ and the claim holds.
\end{proof}

\begin{lem}\noindent \label{S_p_lemma}
\begin{enumerate}
\item Let $R$ be a ring with $p\in \mathbb{Z}$ prime such that $p\in R$ is a non-unit. Let $p\in I\subseteq R$ a proper ideal such that $R/I$ is an $F$-finite ring. Then for all $e\in \mathbb{Z}$, $e>0$, there exists a module finite $R$-algebra $T$ such that $\Gamma_{IT}(R,p^e)\geq 1$.
\item With notation as in (1), assume $(R,\m,k)$ is a complete regular local ring with $k$ $F$-finite. Suppose that $I$ is generated by $\alpha_1,\dots,\alpha_n$ such that $\alpha_1,\dots,\alpha_n$ can be completed to a minimal generating set for $\m$. Then $T$ can be chosen to be regular local with $\alpha_1,\dots,\alpha_n$ part of a minimal generating set of its maximal ideal.
\end{enumerate}
\end{lem}
\begin{proof}\noindent
Set $\bar{R}:=R/I$ and let $F$ denote the Frobenius map on $\bar{R}$.
\begin{enumerate}
\item  By hypothesis, $F^e_*\bar{R}$ is a finite $\bar{R}$-module for all $e$. Taking $T$ to be the $R$-algebra obtained be adjoining $p^e$-th roots of a set of lifts of generators of $F^e_*\bar{R}$ as a $\bar{R}$-module, we have the desired property.
\item Complete $\alpha_1,\dots,\alpha_n$ to a minimal system of generators for $\m$, say $\alpha_1,\dots,\alpha_n, X_{n+1},\dots,X_d$ and let $x_i$ denote the image of $X_i$ in $\bar{R}$. Since $k$ is $F$-finite and $\bar{R}$ is complete, $\bar{R}$ is an $F$-finite regular local ring. Thus, $\bar{R}^{1/p^e}$ is obtained by adjoining to $\bar{R}$, the $p^e$-th roots of the $x_i$ and the $p^e$-th roots of a basis of $F^e_*k$ over $k$. Take $T$ to be the $R$-algebra obtained by adjoining the $p^e$-th roots of the $X_i$ and the $p^e$-th roots of a fixed set of lifts of a minimal generating set of $F^e_*k$ over $k$. By part (1), $\Gamma_{IT}(R,p^e)\geq 1$. Moreover, it follows easily that $T$ is regular local with maximal ideal generated by $(\alpha_1,\dots,\alpha_n,\sqrt[p^e]{X_{n+1}},\dots,\sqrt[p^e]{X_d})$ and residue field $k^{1/p^e}$.
\end{enumerate}
\end{proof}

We include the following for easy reference:

\begin{lem}[\cite{KA}]\label{P6}
Let $S$ be a ring and $p\in S$ a prime integer such that $p$ is a non-unit in $S$. Let $p\geq 3$ and write $p=2k+1$. For $h\in S\setminus pS$ and $x$ an indeterminate over $S$, if
\begin{equation}
    C:=(x-h)^p-(x^p-h^p)=\sum_{j=1}^k (-1)^{j+1}{p\choose j}(x\cdot h)^j[x^{p-2j}-h^{p-2j}]
\end{equation}
$C':=(p(x-h))^{-1}\cdot C$ and $\tilde{P}:=(p,x-h)S[x]$, then $C'\notin \tilde{P}$.
\end{lem}

\begin{lem}[\cite{KA}]\label{l_imodP}
Let $S$ be a ring and $p\in S$ a prime integer such that $p$ is a non-unit in $S$. Let $p\geq 3$ and write $p=2k+1$. For $h\in S\setminus pS$ and $x$ an indeterminate over $S$, suppose $C'$ is as defined in \ref{P6}. Then $C'\equiv h^{p-1}\text{ mod }(p,x-h)S[x]$.
\end{lem}

\section{Roberts's Theorem Revisited}

In this section we give an alternate proof of the main theorem in \cite{RO}. Our extension of this theorem to mixed characteristic relies on the proof in this section. First a preparatory observation.

\begin{prop}\label{prop:RLR_epsilon}
    Suppose $S$ is a regular local ring and $\epsilon$ is a primitive $n$th root of unity for some integer $n$. 
  \begin{enumerate} 
  \item[(i)] If $S$ is unramified, then $S[\epsilon]$ is a regular semi-local ring, and thus a UFD. Moreover, if $S$ is complete and has mixed characteristic, then $S[\epsilon]$ is a regular local ring.
  \item[(ii)] If $n$ is a unit in $S$, then $S[\epsilon]$ is a regular semi-local ring. 
\end{enumerate}
\end{prop} 

\begin{proof}
    For (i) we prove the mixed characteristic case, since the proof of the equi-characteristic case is similar (and easier). So suppose $S$ is an unramified regular local ring of mixed characteristic $p>0$ and $\mathfrak{n}$ its maximal ideal. Write $n= p^rn_0$, with $p \nmid n_0$ and $\epsilon=\epsilon_1\epsilon_2$ where $\epsilon_1$ is a primitive $p^r$-th root of unity and $\epsilon_2$ a primitive $n_0$-th root of unity. By \Cref{prop:cyclotomic_irreducible}, $S[\epsilon_1] = S[x]/(\Phi_{p^r}(X))$. Suppose $M\subseteq S[x]$ is a maximal ideal containing $\n$ and $\Phi_{p^r}(X)$. Since the binomial coefficients ${p^r\choose i}$ for $1\leq i\leq p^r-1$ are divisible by $p$, modulo $p$, we have $x^{p^r}-1=(x-1)^{p^r}$. Since $p\in M$, it follows that $x-1\in M$. Thus, $M = (\n, x-1)S[x]$ is the unique maximal ideal containing $\n$ and $\Phi_{p^r}(X)$, so that $(\n, \epsilon_1-1)S[\epsilon_1]$ is the unique maximal ideal in $S[\epsilon _1]$. In $\mathbb{Z}[\epsilon _1]$, $p = u(\epsilon _1-1)^{\phi (p^r)}$, where $u\in \mathbb{Z}[\epsilon_1]$ is a unit and $\phi(-)$ is the Euler totient function (see \cite[Lemma 10.1]{Neukirch1999}). Thus, $p$ is a redundant generator of $(\n, \epsilon_1-1)S[\epsilon_1]$, so that $S[\epsilon _1]$ is a ramified regular local ring (unless $p=2$ and $r=1$). Now set $T:=S[\epsilon_1]$, so that $S[\epsilon]=T[\epsilon_2]$. Let $\m$ denote the maximal ideal of $T$ and set $k:=T/\m$. Since $T[\epsilon_2]$ is an integral extension of $T$, it is semi-local and each of its maximal ideals contract to $\m$ in $T$.
\par Now $n_0\not\equiv 0$ in $k$, so that the images of $x^{n_0}-1$ and its derivative are relatively prime in $k[x]$. Thus, the image of $x^{n_0}-1$ factors into distinct irreducible factors. Thus, if we write $g(x)$ for the minimal polynomial of $\epsilon_2$ over the quotient field of $T$, then the image of $g(x)$ in $k[x]$, factors into a product of distinct irreducible factors, say $g(x)=q_1(x)\dots q_r(x)+m(x)$, where the images of the $q_i(x)$ in $k[x]$ are the distinct irreducible factors of the image of $g(x)$ in $k[x]$ and $m(x)\in \m[X]$. Since $T[\epsilon_2]=T[x]/(g(x))$, it follows that the maximal ideals of $T[\epsilon_2]$ are $Q_i:=(\m,q_i(\epsilon))S[\epsilon_2]$ for $1\leq i\leq r$. Then, in $T[\epsilon_2]_{Q_1}$, we have $q_1(\epsilon_2)=-(q_2(\epsilon_2)\dots q_r(\epsilon_2))^{-1}m(\epsilon_2)$, so that $Q_1T[\epsilon_2]_{Q_1}=\m T[\epsilon_2]_{Q_1}$. Thus, $T[\epsilon_2]_{Q_1}$ is a regular local ring. The argument is similar for $i=2,\dots,r$. Therefore $T[\epsilon_2]$ is a regular semi-local ring. Since a semi-local domain which is locally a UFD, is a UFD, we have that $T[\epsilon_2]=S[\epsilon]$ is a UFD.

\medskip
\noindent
Part (ii) follows in the same way as the $n_0$ case above. 
\end{proof}

\begin{rem}\label{rem:primitive_root_of unity}
     If $(S, \n, k)$ is a ramified regular local ring of mixed characteristic $p$, then for $\epsilon$ a primitive $p$-th root of unity, $S[\epsilon]$ need not be regular. For instance, set $S$ to be $V[x,y]/(p-x^2y^3)$ localized at the ideal generated by the images of $p, x, y$ where $V$ is any DVR with uniformizing parameter $p$. Then $S$ is a ramified regular local ring. Let $\epsilon$ be a primitive $p$th root of unity. Let $\m$ be a maximal ideal in $S[\epsilon]$. Let the polynomial ring $S[t]$ map onto $S[\epsilon]$ in the obvious way. Then $\m$ corresponds to a maximal ideal $M\subseteq S[t]$ containing $p$ and $t^p-1$, and hence $M = (x, y, t-1)$ and thus $\m = (x, y, \epsilon-1)S[\epsilon]$ is the unique maximal ideal of $S[\epsilon]$. In $S[\epsilon]$ we have $x^2y^3 = p = u(\epsilon-1)^{p-1}$, for $u\in \mathbb{Z}[\epsilon]$ a unit. Thus, $S[\epsilon]$ is not a regular local ring. 
\end{rem}

\begin{thm}[\textbf{Roberts's Theorem \cite{RO}}]\label{thm:Roberts}
   Let $(S, \n, k)$ be a regular local ring with quotient field $L$ and $R$ the integral closure of $S$ in a finite abelian extension $L\subseteq K$. Assume the characteristic of $k$ does not divide $[K:L]$. Then $R$ is Cohen-Macaulay.  
\end{thm}
\begin{proof}
    Let $n$ denote the order of the Galois group of $K$ over $L$, so that $n$ is a unit in $S$. Let $\epsilon$ be a primitive $n$-th root of unity. Then by \Cref{prop:RLR_epsilon}, $S[\epsilon]$ is a (possibly ramified) regular semi-local ring. Moreover, $R[\epsilon]$ is the integral closure of $S[\epsilon]$ in $K(\epsilon)$. To see this, it is enough to show that $R[\epsilon]$ is integrally closed. Let $R_1$ be the integral closure of $R$ in $K(\epsilon)$ and $f(x)\in R[x]$, the minimal polynomial of $\epsilon$ over $K$. Since $n$ is a unit in $R$, $x^n-1$ has distinct roots and hence $f(x)$ is a separable polynomial. Thus, $f'(\epsilon)R_1\subseteq R[\epsilon]$. Since $R$ is integrally closed, $x^n-1 = f(x)g(x)$, with $g(x)\in R[x]$. Thus, $n\epsilon^{n-1} = f'(\epsilon)g(\epsilon) $, so that $nR_1\subseteq R[\epsilon]$. Since $n$ is a unit in $R[\epsilon]$, we have $R[\epsilon] = R_1$. 
\par
Suppose we could show that $R[\epsilon]$ is Cohen-Macaulay. Since $R[\epsilon]$ is free over $R$, $R$ is a summand of $R[\epsilon]$, and thus $R$ is Cohen-Macaulay. Therefore, it remains to be seen that $R[\epsilon]$ is Cohen-Macaulay. For this, we use Kummer Theory.
Now, it is straightforward to see that $\mathrm{Gal} (K(\epsilon)/L(\epsilon))$ is isomorphic to a subgroup of $\mathrm{Gal} (K/L)$, and hence $\mathrm{Gal} (K(\epsilon)/L(\epsilon))$ is an abelian group. Let $t$ denote the exponent of $\mathrm{Gal}(K(\epsilon)/L(\epsilon))$, so that $t$ is a unit in $S[\epsilon]$, since $t\mid n$. By Kummer theory, there exist $a_1, \ldots, a_s\in L(\epsilon)$ such that $K(\epsilon) = L(\epsilon, \sqrt[t]{a_1}, \ldots, \sqrt[t]{a_s})$. Clearing denominators, we may assume that each $a_i\in S[\epsilon]$. Thus, $K(\epsilon)$ is the quotient field of $S[\epsilon][\sqrt[t]{a_1}, \ldots, \sqrt[t]{a_s}]$. 

\par
Now, as an element of $S[\epsilon]$, each $a_i$ is a unit times a product of primes. Let $q_1, \ldots, q_h$ be the distinct unit and prime factors appearing among $a_1, \ldots, a_s$. Then no height one prime of $S[\epsilon]$ contains any two $q_i, q_j$. Thus, by \Cref{HK2}, $T := S[\epsilon][\sqrt[t]{q_1}, \ldots, \sqrt[t]{q_h}]$ is integrally closed. Set $E$ to be the quotient field of $T$. Moreover, $K(\epsilon)\subseteq E$, so that $R[\epsilon] \subseteq T$. By \Cref{lem:deg_extension}, degree of $E$ over $L(\epsilon)$ is a is a unit in $R[\epsilon]$. Therefore, the degree of $E$ over $K(\epsilon)$ is a unit in $R[\epsilon]$, and hence $R[\epsilon]$ is a summand of $T$ via the splitting given by restricting the field trace map and dividing by the degree of $E$ over $K(\epsilon)$. But $T$ is a free extension of $S[\epsilon]$, so $T$ is Cohen-Macaulay, and hence $R[\epsilon]$ is Cohen-Macaulay, which completes the proof.
\end{proof}

\section{Abelian extensions and the tamely $p$-ramified property}\label{sec:Abelian_extensions&p_unramified_property}

The goal of this section is twofold: firstly, to characterize the tamely $p$-ramified property in terms of a certain numerical criterion in codimension one and secondly, to define tamely $p$-ramified generically abelian extensions of an unramified regular local ring.

\par
We will show

\begin{thm}\label{thm:p_ramification_characterization}
  Let $S$ be a regular local ring such that $\mathrm{char}(\mathrm{Frac}(S))=0$. Assume $S$ possesses a primitive $p$-th root of unity for $p\in \mathbb{Z}$ a prime integer. Then the following are equivalent:
  \begin{enumerate}
      \item $0\neq f\in S$ is $p$-unramified.
      \item $0\neq f\in S$ is tamely $p$-ramified.

    \item either
    \begin{enumerate}
      \item $f\in S$ is a $p$-th power or
      \item $f\notin \bigcup_{Q\in \ass(S/(p))} Q$ and for all $Q\in \ass(S/(p))$

   \[\Gamma_{QS_Q}(f)\geq [\sum_{i=0}^{\infty}(1/p^i)]ord_Q(p)=\dfrac{p}{p-1}ord_Q(p).\]
  \end{enumerate}
    
  \end{enumerate} 

\end{thm}

\begin{lem}\label{CM_dual_prop}
Let $\D$ be a Gorenstein Noetherian domain such that the prime integer $p$ is a non-unit in $\D$. Suppose that $Ass(\D/(p))=\{(\alpha)\}$ and let $k$ be such that $p\in (\alpha^k)\setminus (\alpha^{k+1})$. Let $f\in \D$ be such that $f$ is not a $p$-th power and $1\leq q:=\Gamma_{(\alpha)}(f)$. Let $\omega$ be a root of the monic polynomial $X^p-f\in \D[X]$ in some algebraic closure of the fraction field of $\D$. Set $r=min\{q,k+1\}$. For any positive integer $n$ satisfying $n\leq (p-1)^{-1}min\{q,k\}$ and $h$ any $p$-th root of $f$ modulo $\alpha^r$, set $J_{n,h}:=(\omega-h,\alpha^n)^{p-1}\D[\omega]$. Then
\begin{enumerate}
    \item $Hom_{\D[\omega]}(J_{n,h},\D[\omega])=\langle 1,q_1,\dots,q_{p-1}\rangle_{\D[\omega]}$ where $q_i:J_{n,h}\rightarrow \D[\omega]$ is the map given by multiplication by $\alpha^{-ni}(\omega-h)^i$.
    \item $J_{n,h}$ is $P$-primary for $P:=(\alpha, \omega-h)$, the unique associated prime of $p\D[\omega]$.
    \item $Hom_{\D[\omega]}(J_{n,h},\D[\omega])$ is a maximal Cohen-Macaulay $\D$-module.
\end{enumerate}
\end{lem}
\begin{proof}
Since $f\in \D$ is not a $p$-th power, we have $\D[\omega]\simeq \dfrac{\D[X]}{(X^p-f)}$. Write $f=h^p+\alpha^r\cdot b$ for some $h,b\in \D$. Taking $S=\D$ in \Cref{P6}, let $C'\in \D[X]$ be as in \Cref{P6}. We have
 \begin{align}\label{F(W)_containment}
 X^p-f &=X^p-h^p-\alpha^{r}b \nonumber
 \\
 &= (X-h)(X^{p-1}+\dots+h^{p-1})-\alpha^{r}b \nonumber
 \\
 &= (X-h)((X-h)^{p-1}+C'p)-\alpha^{r}b \nonumber
 \\
 &= (X-h)\cdot (X-h)^{p-1}+C'p(X-h)-\alpha^{r}b \nonumber
 \\
 &= (X-h)\cdot (X-h)^{p-1} +\alpha^{n(p-1)}\cdot \gamma 
 \end{align}
for some $\gamma\in \D[X]$. Thus $X^p-f\in \tilde{J}_{n,h}:=(X-h,\alpha^n)^{p-1}\subseteq \D[X]$. Since $\tilde{J}_{n,h}$ is a power of a complete intersection ideal, it is unmixed. Moreover, it is $\tilde{P}$-primary for $\tilde{P}:=(X-h,\alpha)$. Thus $J_{n,h}$ is $P$-primary and (2) holds.
\par $\tilde{J}$ is the ideal of maximal minors of the $p\times (p-1)$ matrix 
 \[
\Phi_{n,h} = \begin{bmatrix} 
    X-h & 0 & \dots & 0 & 0 \\
    \alpha^n & X-h & 0 &\dots & 0\\
    0  & \alpha^n &X-h  &\dots & 0  \\
    \vdots &0 &\ddots & \ddots& \vdots\\
    0 & \dots & 0 & \alpha^n & X-h \\
    0 & \dots  & 0 & 0 & \alpha^n 
    \end{bmatrix}.
    \] 
    Set
   \[
\Psi_{n,h} = \begin{bmatrix} 
    X-h & 0 & \dots & 0 & 0& -\gamma\\
    \alpha^n & X-h & 0 &\dots & 0& 0\\
    0  & \alpha^n  &X-h  &\dots & 0 & 0 \\
    \vdots &0 &\ddots & \ddots& \vdots & \vdots\\
    0 & \dots & 0 & \alpha^n & X-h&0 \\
    0 & \dots  & 0 & 0 & \alpha^n & -(X-h)
    \end{bmatrix}.
\]  
By \cite[Lemma 2.5]{KU}, $Hom_{\D[\omega]}(J_{n,h},\D[\omega])\simeq (J_{n,h})^{-1}_{\D[\omega]}$ is generated as a $\D[\omega]$-module by the maps given by multiplication by the elements $\{\delta_i^{-1}\psi_{i,i}\}_{1\leq i\leq p}$ where $\delta_i$ and $\psi_{i,i}$ denote the images in $\D[\omega]$ of the $i$-th signed minor of $\Phi_{n,h}$ and the $(i,i)$-th cofactor of $\Psi_{n,h}$ respectively. This is precisely the generating set claimed in (1).
\par For (3), note that $Hom_{\D[\omega]}(J_{n,h},\D[\omega])\simeq (J_{n,h})^{-1}_{\D[\omega]}$ is isomorphic to a $1$-link of $J_{n,h}$. Since $J_{n,h}$ is height one unmixed and $\D[\omega]$ is Gorenstein, $\D[\omega]/J_{n,h}$ is Cohen-Macaulay if and only if $\D[\omega]/I$ is Cohen-Macaulay for any $1$-link $I$ of $J_{n,h}$ (see \cite{10.2307/1971402}[Proposition 2.5]). Moreover, $I$ is a Cohen-Macaulay $\D[\omega]$-module if and only if $\D[\omega]/I$ is a Cohen-Macaulay ring. Since $\D[\omega]/J_{n,h}\simeq \D[X]/\tilde{J}_{n,h}$ is Cohen-Macaulay by the Hilbert-Burch theorem, we are done.
\end{proof}

\begin{lem}\label{lem:forward_implication}
Let $S$ be a regular local ring such that $\mathrm{char}(\mathrm{Frac}(S))=0$. Assume $S$ possesses a primitive $p$-th root of unity for $p\in \mathbb{Z}$ a prime integer. If $0\neq f\in S$, $f\notin \bigcup_{Q\in \ass(S/(p))} Q$, satisfies for all $Q\in \ass(S/(p))$, $\Gamma_{QS_Q}(f)\geq p\cdot(p-1)^{-1}\mathrm{ord}_Q(p)$, then
\begin{enumerate}
    \item $S\to \overline{S[\omega]}$ is \'{e}tale over $p$ in codimension one for some (equivalently every) root $\omega$ of the polynomial $W^p-f\in S[W]$, where $\overline{*}$ denotes normalization and $W$ is an indeterminate over $S$.
    \item If $p\in S$ is not a unit, then for each prime divisor $q$ of $p$, $\overline{S[\omega]}_{(q)}$ is generated as a $S[\omega]_{(q)}$ module by $\{1,q^{-r}(\omega-h_q),\dots,q^{-r(p-1)}(\omega-h_q)^{p-1}\}$ for some $h_q\in S_{(q)}$ and $r=(p-1)^{-1}\mathrm{ord}_q(p)\in \mathbb{Z}$.
    \end{enumerate}
\end{lem}

\begin{proof}
    (1) holds vacuously if $p\in S$ is a unit. Assume $p\in S$ is not a unit. Let $\omega$ be any root of the polynomial $W^p-f\in S[W]$ in a fixed algebraic closure of $\mathrm{Frac}(S)$ and set $R=\overline{S[\omega]}$. Let $q$ be any prime divisor of $p$. By \Cref{rem:normal_etale}, it suffices to show $S_q\to R_q$ is unramified. Write $f=h_q^p+q^nb_q$ for some $h_q,b_q\in S_{(q)}$ and $n\geq p/p-1\cdot \mathrm{ord}_Q(p)$. Since we work locally, we drop the lower index ``$q$". We have in $A:=S_{(q)}[\omega]$ :  

\begin{align}\label{eq1_Dw}
 q^{n}\cdot b &=\omega^p -h^p \nonumber
 \\
 &= (\omega-h)^p+p\cdot c'\cdot (\omega-h)
 \end{align}
where $c'\in A$ is the image in $A$ of the element $C'\in S_{(q)}[W]$ in \Cref{P6}. Let $\epsilon\in S$ be a primitive $p$-th root of unity. It is standard to see that $p=-(c'_{\epsilon})^{-1}(\epsilon-1)^{p-1}$ where $c'_{\epsilon}$ is the image in $\mathbb{Z}[\epsilon]$ of the corresponding element $C'_{\epsilon}\in  \mathbb{Z}[W]$ from \Cref{P6}. Note that this implies $\epsilon-1\in qS$ and $\mathrm{ord}_q(\epsilon-1)\cdot (p-1)=\mathrm{ord}_q(p)$. We have
\begin{align}\label{main_eq_Dw}
 (\omega-h)^p-((c'_{\epsilon})^{-1}\cdot c')(\omega-h)(\epsilon-1)^{(p-1)}-q^nb=0.
 \end{align}
Let $(\epsilon-1)=\mu\cdot q^r$ for some unit $\mu\in S_{(q)}$, so that $\mathrm{ord}_q(\epsilon-1)=r$. Setting $U:=(\omega-h)$ and $V:=q^r$, \Cref{main_eq_Dw} looks like
 \begin{align}\label{main_homogeneous}
 U^p-((c'_{\epsilon})^{-1}\cdot c'\cdot \mu^{p-1})UV^{p-1}-V^pb'=0 
 \end{align}
for some $b'\in A$. Thus $V^{-1}U$ is integral over $A$
 \begin{align}
 (U/V)^p-((c'_{\epsilon})^{-1}\cdot c'\cdot \mu^{p-1})(U/V)-b'=0.
 \end{align}
 We claim that $B:=A[U/V]$ is regular. Setting $P:=(q,\omega-h)A$, maximal ideals in $B$ correspond to height two primes in $A[X]/PA[X]$ containing the image of $Ker(\phi)$ where $\phi:A[X]\rightarrow A[U/V]$ is the natural $A$-algebra map. From, \Cref{l_imodP} we have
 \begin{align}\label{main_homogeneous_modp}
 X^p-((c'_{\epsilon})^{-1}\cdot c'\cdot \mu^{p-1})X-b'\equiv X^p-(h\mu)^{p-1}X-b'\;mod\:(PA[X]).
 \end{align}
By hypothesis, $h\in A$ is a unit and hence the image of \Cref{main_homogeneous_modp} in $(A/P)[X]$ is irreducible over $A/P$ if and only if the Artin-Schreier polynomial $(X/h\mu)^p-(X/h\mu)-b'/(h\mu)^p$ is irreducible. If this polynomial is irreducible, then since $S_{(q)}$ is universally catenary, it follows that there exists a unique maximal ideal $Q$ of $B$ satisfying $QB_Q=PB_Q=qB_Q$. Moreover, the associated extension of residue fields is separable since it is given by a Galois extension of order $p$. If the polynomial is reducible, then necessarily
 \begin{align}\label{main_homogeneous_modp_2}
X^p-((c'_{\epsilon})^{-1}\cdot c'\cdot \mu^{p-1})X-b'\equiv \prod_{i=1}^p(X-\beta-ih\mu)\;mod\:(PA[X])
 \end{align}
 for some $\beta\in A$. Therefore, there are exactly $p$ maximal ideals in $B$, say $Q_1,\dots,Q_p$, satisfying $Q_iB_{Q_i}=(P,U/V-\beta-ih)B_{Q_i}=(\alpha)B_{Q_i}$. Moreover, the associated extension of residue fields in each case is trivial and in particular separable. Thus $B$ is regular and (1) holds. The assertion (2) also follows from what we have just showed.
\end{proof}

\begin{proof}[Proof of Theorem-\ref{thm:p_ramification_characterization}]

  Since $S$ possesses a primitive $p$-th root of unity, if $f\in S$ is a $p$-th power, then $\overline{S[\omega]}=S$, so trivially the associated extension is $p$-unramified. The remainder of the implication (3) $\implies$ (1) is the content of \Cref{lem:forward_implication}(1). The implication (1) $\implies (2)$ is trivial. We now prove the contrapositive of (2) $\implies$ (3). Let $0\neq f\in S$ and $\omega$ any root of the polynomial $W^p-f\in S[W]$ in a fixed algebraic closure of $L:=\mathrm{Frac}(S)$. Assume $f\in S$ is not a $p$-th power. Set $R=\overline{S[\omega]}$. Suppose there exists a prime divisor $q\in S$ of $p$ such that $\Gamma_{qS_{(q)}}(f)<\dfrac{p}{p-1}\mathrm{ord}_q(p)$. We will show that $S\to \overline{S[\omega]}$ is not tamely ramified over $p$. If $\Gamma_{qS_{(q)}}(f)=0$, then $S_{(q)}\to S_{(q)}[\omega]$ induces a purely inseparable extension of residue fields. Hence $S_{(q)}\to R_{(q)}$ is not tamely ramified over $p$. Now assume $\Gamma_{qS_{(q)}}(f)=n\geq 1$. Let $\epsilon\in S$ be a primitive $p$-th root of unity. It is standard to see that $p=-(c'_{\epsilon})^{-1}(\epsilon-1)^{p-1}$ where $c'_{\epsilon}$ is the image in $\mathbb{Z}[\epsilon]$ of the corresponding element $C'_{\epsilon}\in  \mathbb{Z}[W]$ from \Cref{P6}. Note that this implies $\epsilon-1\in qS$ and $\mathrm{ord}_q(\epsilon-1)\cdot (p-1)=\mathrm{ord}_q(p)$. Write $f=h^p+q^nb$ for some $b,h\in S_{(q)}$. Let $(\epsilon-1)=\mu\cdot q^r$ for some unit $\mu\in S_{(q)}$, so that $\mathrm{ord}_q(\epsilon-1)=r$. As in the proof of \Cref{lem:forward_implication} we have in $A:=S_{(q)}[\omega]$:

    \begin{align}
 (\omega-h)^p-((c'_{\epsilon})^{-1}\cdot c'\cdot \mu^{p-1})(\omega-h)q^{r(p-1)}-q^nb=0.
 \end{align} 
Write $n=kp+s$ with $0\leq s<p$. Note that $kp\leq n<pr$ implies $k<r$. Dividing by $q^{kp}$ and setting $\beta=q^{-k}(\omega-h)\in L(\omega)$, we have

 \begin{align}\label{eq:beta}
 \beta^p-(c'_{\epsilon})^{-1}\cdot c'\cdot (\mu q^{r-k})^{p-1}\beta-q^sb=0 
 \end{align} 
an integral equation for $\beta$ over $A$. Suppose $s=0$. We claim that this implies $\Gamma_{qS_{(q)}}(b)=0$. To see this , let if possible $b=u^p+qv$ for some $u,v\in S_{(q)}$. Then $f=h^p+q^{kp}(u^p+qv)=(h+q^ku)^p+q^{kp+1}v+pq^ku\alpha$ for some $\alpha\in S_{(q)}$. Now, $k+\mathrm{ord}_q(p)>k+\dfrac{(p-1)n}{p}=n/p+(p-1)n/p=n$. This implies $\Gamma_{qS_{(q)}}(f)\geq n+1$, a contradiction. Thus if $s=0$, then $\Gamma_{qS_{(q)}}(b)=0$. Noting that the extension $S_{(q)}\to S_{(q)}[\omega]$ induces a trivial extension of residue fields, one sees that if $s=0$, then $A[\beta]$ is local and the extension $S_{(q)}\to A[\beta]$ induces a purely inseparable extension of residue fields. Thus, if $s=0$, $S_{(q)}\to R_{(q)}$ is not tamely ramified over $p$. Now, assume $s\geq 1$. From \Cref{eq:beta}, one sees that in the local ring $A[\beta]$, $q^s$ is an associate of $\beta^p$. This continues to hold after localization at each maximal ideal of $\overline{S[\omega]}$. Thus $p$ divides the order of $q^s$ at each maximal ideal of $\overline{S[\omega]}$. Since $0\leq s<p-1$, this implies $p$ divides the order of $q$ at each maximal ideal of $\overline{S[\omega]}$. Thus $S_{(q)}\to R_{(q)}$ is not tamely ramified over $p$.
 
  Finally, suppose $f$ is divisible by some prime divisor $q$ of $p$. We may assume that $\mathrm{ord}_q(f)\leq p-1$. It is then easily seen that $\Gamma_{qS_{(q)}}(f)\leq p-1 <\dfrac{p}{p-1}\mathrm{ord}_q(p)$, so that by what we have just shown $S_{(q)}\to R_{(q)}$ is not tamely ramified over $p$. Thus the proof of the implication (2) $\implies$ (3) and the theorem is complete.
\end{proof}

We now proceed to define tamely $p$-ramified and tamely $p$-ramified type abelian extensions of an unramified regular local ring of mixed characteristic $p>0$. Let $S$ denote a semi-local regular ring such that $L:=\mathrm{Frac}(S)$ has characteristic zero. Assume that $S$ possesses a primitive $n$-th root of unity for an integer $n$. Let $K/L$ be a finite abelian extension of exponent $n$. By Kummer theory, $K$ corresponds to a unique subgroup $U$ of $L^{\times}$ containing $L^{\times n}$ as a subgroup of index equal to $[K:L]$. Since $S$ is a unique factorization domain, each coset of $U$ in turn corresponds to a unique element of the monoid $S/S^n$. For each non trivial class in  $S/S^n$, there is a unique element (up to multiplication by an element of $(S^{\times})^n$) in $S$ representing it satisfying the condition that its order at each of its prime divisors is at most $n$. We call such a representative in $S$ corresponding to a coset $U$ a \textit{canonical element} for $K/L$ in $S$. The \textit{canonical divisors} of $K/L$ in $S$ are the union of the prime divisors of each canonical element (possibly empty). \footnote{Our usage of the term canonical divisor is specific to our setting and is not meant to suggest any connection with its typical usage in algebraic geometry.}

\begin{ex}
    Suppose $K/L$ above has order $p$ for $p\in \mathbb{Z}$ a prime integer. Then $K=L(\omega)$ for $\omega$ a root of the irreducible polynomial $W^p-f\in L[W]$, for $W$ an indeterminate over $L$. Moreover, $K=L(\mu)$ for $\mu$ a root of $W^p-g\in L[W]$ if and only if $g\in \cup_{i=0}^{p-1} f^i\cdot (L^{\times})^p$. There is a unique element (up to multiplication by an element of $(S^{\times})^p$) $g\in (f\cdot (L^{\times})^p)\cap S$ such that $g=\lambda a_1^{e_1}\dots a_r^{e_r}$, where $\lambda\in S$ is a unit, $a_1,\dots,a_r\in S$ are primes and $0\leq e_1,\dots,e_r\leq p-1$. Then $g$ is a canonical element for $K/L$ in $S$ corresponding to the coset $f\cdot (L^{\times})^p$ and is unique up to multiplication by an element of $(S^{\times})^p$. The canonical elements corresponding to the cosets $f^i\cdot L^p$ divide a sufficiently large power of $g$ for each $0\leq i\leq p-1$. Hence the canonical divisors of $K/L$ in $S$ are $a_1,\dots,a_r$ (possibly empty).
\end{ex}

\begin{ex}
  Now let $K/L$ be a finite abelian extension with Galois group $G:=\mathbb{Z}/p\mathbb{Z}^{\oplus n}$ for $p\in \mathbb{Z}$ a prime integer. Let $K_1,\dots, K_n$ denote the fixed fields of the $n$ subgroups of $G$ obtained by taking the quotient by each copy of $\mathbb{Z}/p\mathbb{Z}$. The canonical divisors of $K/L$ are the union of the divisors of each $K_i/L$.
\end{ex}

Now let $R$ be the integral closure of $S$ in $K$. Let $X\subseteq \spec(S)$ be the ramification locus of the map $S\to R$ and let $X_0\subseteq X$ the subset of codimension one primes. Then $V(X_0)=X$. To see this, note that a map of rings $A\to B$ is unramified if and only if it is of finite type and the module of \"{K}ahler differentials $\Omega_{B/A}=0$. Since $\Omega_{B/A}$ commutes with localization, it follows that $V(X_0)\subseteq X$. The reverse inclusion $X\subseteq V(X_0)$ follows from the purity of branch locus. For each prime integer divisor $q$ of $n$, let $\mathfrak{X}_q$ denote the prime divisors of $q$ in $S$, and $\mathfrak{X}$ the union of all the $\mathfrak{X}_q$.

\begin{prop}\label{prop:ramification_locus}
  Let $\mathfrak{C}$ denote the canonical divisors of $K/L$ in $S$. We then have
    \[X_0\backslash \mathfrak{X}\subseteq \mathfrak{C}\subseteq X_0\subseteq \mathfrak{C}\cup \mathfrak{X}.\]
\end{prop}

\begin{proof}
    Consider a height one prime $Q$ of $S$ outside $\mathfrak{C}\cup \mathfrak{X}$. By Kummer theory and \Cref{HK2}, $R_Q=S_Q[\sqrt[n]{f_1},\dots,\sqrt[n]{f_r}]$ where $f_1,\dots,f_r$ are the canonical elements of $K/L$ in $S$. Since $Q$ is coprime to the canonical divisors of $K/L$ in $S$ and the prime divisors of $n$, it follows that the induced extensions of residue fields are all separable. From \Cref{HK1}, it then follows that $S_Q\to R_Q$ is unramified. This gives the inclusions $X_0\backslash \mathfrak{X}\subseteq \mathfrak{C}$ and $X_0\subseteq \mathfrak{C}\cup \mathfrak{X}$. Now clearly if $c\in \mathfrak{C}$ is a divisor of a canonical element $f$, then $c$ ramifies in the extension $S\to \overline{S[\sqrt[n]{f}]}$ and hence ramifies in $S\to R$, i.e., $c\in X_0$. This completes the proof.  
\end{proof}

\begin{rem}
    If $S$ is local and $n\in S$ is a unit (in particular under the setup of \Cref{thm:Roberts}), \Cref{prop:ramification_locus} is saying $\mathfrak{C}=X_0$, i.e., the canonical divisors of $K/L$ in $S$ are precisely the height one primes in $S$ that ramify along $S\to R$.
\end{rem}

\begin{dfn}\label{def:p-unramified_Abelian}
  Let $S$ be a semi-local regular ring with $L:=\mathrm{Frac}(S)$ having characteristic zero and $p\in \mathbb{Z}$ a prime integer. Assume $S$ possesses a primitive $n$-th root of unity and let $K/L$ be a finite abelian extension whose Galois group has exponent $n$. Let $R$ be the integral closure of $S$ in $K$.  We say $K/L$ is \textit{tamely $p$-ramified over $S$} if $S\to R$ is tamely $p$-ramified and either the set of canonical divisors are empty or the canonical divisors are tamely $p$-ramified over $n$.
    \par By construction, the set of canonical divisors for $K/L$ in $S$ are uniquely determined, so the property of being a tamely $p$-ramified abelian extension is intrinsic to the given extension.
\end{dfn}

\begin{rem}\label{rem:canonical_divisors_modular_setup}
Assume notation as in \Cref{def:p-unramified_Abelian}. Suppose $S$ is local of mixed characteristic $p$ and $p\mid n$. Then if $S\to R$ is tamely $p$-ramified (in particular if $K/L$ is tamely $p$-ramified over $S$), we have $\mathfrak{C}=X_0\backslash V(p)$, i.e., the canonical divisors of $K/L$ in $S$ are precisely the codimension one primes in $S$ away from $p$ that ramify in $R$. To see this, note that in \Cref{prop:ramification_locus}, $\mathfrak{X}$ is nothing but the prime divisors of $p$ in $S$. Thus, $X_0\backslash V(p)\subseteq \mathfrak{C}$ by \Cref{prop:ramification_locus}. Now let if possible $q\in \mathfrak{C}$ for $q$ a prime divisor of $p$ in $S$. Then if $q$ is a divisor of a canonical element $f\in S$, one sees that $S\to \overline{S[\sqrt[p]{f}]}$ is not tamely $p$-ramified. Hence $S\to \overline{S[\sqrt[n]{f}]}$ and $S\to R$ are not tamely $p$-ramified. This is a contradiction. Thus $\mathfrak{C}\cap V(p)=\emptyset$ and we see from \Cref{prop:ramification_locus} that $\mathfrak{C}\subseteq X_0\backslash V(p)$.
\end{rem}

\begin{dfn}\label{def:quasi_p_unramified}
    Let $S,L,K$ and $p$ be as in \Cref{def:p-unramified_Abelian}. We say $K/L$ is \textit{of tamely $p$-ramified type over $S$} if there exists a module finite injective map of regular rings $f:S\to T$ such that 

    \begin{enumerate}
        \item $f_*:\spec(T)\to \spec(S)$ is injective in codimension one on the fiber over $V(p)$ and
        \item for all $Q\in V(p)\subseteq \spec(S)$, $ord_Q(p)=ord_{f_*^{-1}(Q)}(p)$ and
        \item $\mathrm{Frac}(T)K/\mathrm{Frac}(T)$ is tamely $p$-ramified over $T$.
    \end{enumerate}
\end{dfn}

\begin{ex}
    Let $L=\mathbb{Q}(X)$ for $X$ an indeterminate over $\mathbb{Q}$ and $K:=L(\omega)$ for $\omega$ a root of the polynomial $Y^2-X-4\in L[Y]$. Consider the two dimensional regular local subring $S\subseteq L$ defined as $S:=\mathbb{Z}[X]_{(2,X)}$. Then $K/L$ is not tamely $2$-ramified over $S$. To see this, note that $S[\omega]$ is integrally closed and that that $S\to S[\omega]$ is not \'{e}tale in codimension one over $2$.  Consider the extension of regular local rings $S\to T$ for $T:=\mathbb{Z}[\sqrt{X}]_{(2,\sqrt{X})}$. Then $\mathbb{Q}(\sqrt{X})(\omega)/\mathbb{Q}(\sqrt{X})$ is tamely $2$-ramified over $T$ by \Cref{thm:p_ramification_characterization}. Thus, $K/L$ is of tamely $2$-ramified type over $S$.
\end{ex}

Note that if $S$ is an unramified regular local ring of mixed characteristic $p>0$, then $S$ does not possess a primitive $p$-th root of unity unless $p=2$.

\begin{dfn}\label{p-unramified_unramified_RLR}
    Let $S$ be an unramified regular local ring with $L:=\mathrm{Frac}(S)$ having characteristic zero and $p\in \mathbb{Z}$ a prime integer. Let $K/L$ be a finite abelian extension with exponent $n$. Then $K/L$ is \textit{(resp. of tamely $p$-ramified type) tamely $p$-ramified} over $S$ if $K(\epsilon)/L(\epsilon)$ is (resp. of tamely $p$-ramified type) tamely $p$-ramified over $S[\epsilon]$ for some (equivalently every) primitive $n$-th root of unity $\epsilon$.
\end{dfn}

\begin{rem}
    Note that \Cref{p-unramified_unramified_RLR} is well defined due to \Cref{prop:RLR_epsilon} and the fact that if $\epsilon$ and $\epsilon'$ are distinct primitive $n$-th roots of unity in a fixed algebraic closure of $L$, then $S[\epsilon]=S[\epsilon']$.
\end{rem}

\begin{ex}\label{ex:Roberts}
    Assume notation as in \Cref{thm:Roberts}. In addition, assume $S$ is an unramified regular local ring of mixed characteristic $p>0$. Then $K/L$ is tamely $p$-ramified over $S$. To see this, let $n$ be the exponent of $\mathrm{Gal}(K/L)$ and $\epsilon$ a primitive $n$-th root of unity. Let the exponent of $\mathrm{Gal}(K(\epsilon)/L(\epsilon))$ be $m$. If $R$ is the integral closure of $S$ in $K(\epsilon)$, then as in the proof of \Cref{thm:Roberts}, one constructs a small Cohen-Macaulay algebra $R\to T$. By construction, it follows $S[\epsilon]\to T$ is tamely $p$-ramified, hence, $S[\epsilon]\to R$ is tamely $p$-ramified. Moreover, if $f\in S[\epsilon]$ is a canonical divisor for $K(\epsilon)/L(\epsilon)$ and is not equal to $\epsilon^{n/p}-1$, then $S[\epsilon][\sqrt[m]{f}]$ is integrally closed by \Cref{HK2}. Since $m$ is coprime to $p$ it follows that $S[\epsilon]\to S[\epsilon][\sqrt[m]{f}]$ is $p$-unramified (in particular tamely $p$-ramified). If $\epsilon^{n/p}-1$ is a canonical divisor, then $S[\epsilon,\sqrt[m]{\epsilon^{n/p}-1}]=S[\sqrt[m]{\epsilon^{n/p}-1}]$ is regular local and $S[\epsilon]\to S[\sqrt[m]{\epsilon^{n/p}-1}]$ is tamely $p$-ramified.
\end{ex}

\section{Abelian extensions with $p$-torsion annihilated by $p$}

In this section, we prove \Cref{thm:roberts_extension}, an extension of Roberts's theorem to mixed characteristic. We work under the hypothesis that the base regular ring is unramified. Our arguments use this assumption in an essential way: for instance to preserve regularity upon adjunction of a primitive $p$-th root of unity, see \Cref{rem:primitive_root_of unity}. Note that in this case, under the assumptions of \Cref{thm:Roberts} (i.e. the non-modular case), a generically abelian extension is automatically tamely $p$-ramified and the $p$-torsion in the Galois group is annihilated by $p$. The former follows from \Cref{ex:Roberts} and the latter is trivial, since there is no $p$-torsion in the Galois group. Thus the following theorem is an extension of Roberts's theorem to mixed characteristic $p>0$. We also illustrate our results by discussing how they apply to an example of Koh in \cite{KO} that exhibits the failure of Roberts's theorem in the modular case.

\begin{thm}\label{thm:roberts_extension}
Let $S$ be an unramified regular local ring of mixed characteristic $p>0$ with quotient field $L$. Let $K/L$ be a finite abelian extension with $p$-torsion of the Galois group $\mathrm{Gal}(K/L)$ annihilated by $p$ and $R$ the integral closure of $S$ in $K$. If $K/L$ is \textit{of tamely $p$-ramified type over $S$}, then $R$ admits a small Cohen-Macaulay algebra.
\end{thm}

\begin{proof}
Let $n$ denote the exponent of $K/L$. Fix an algebraic closure of $L$ and let $\epsilon$ be any primitive $n$-th root of unity in it. Choose a map of regular rings $S[\epsilon]\to T$ satisfying the conditions of \Cref{def:quasi_p_unramified} ($S[\epsilon]$ is regular by \Cref{prop:RLR_epsilon}). Set $\psi=\epsilon^{n/p}$. Note that $\psi-1$ is the unique prime divisor of $p$ in $S[\psi]$ and $\mathrm{ord}_{\psi-1}(p)=p-1$. Note that $S[\psi]\to S[\epsilon]$ is \'{e}tale. By purity of branch locus, it suffices to check this in codimension one. Let $F\in S[\psi][X]$ be the minimal polynomial of $\epsilon$ over $L(\psi)$. Note that $n/p$ is a unit in $S[\psi]$. Since $F$ divides $X^{n/p}-1\in L[X]$ and the latter is separable modulo any height one prime of $S[\psi]$, it follows that $S[\psi]\to S[\epsilon]$ is \'{e}tale in codimension one. Thus, $\psi-1$ is the unique prime divisor of $p$ in $S[\epsilon]$ and $\mathrm{ord}_{(\psi-1)S[\epsilon]}(p)=p-1$. These continue to hold in $T$ due to conditions (1) and (2) in \Cref{def:quasi_p_unramified}. 

\par Let $\mathfrak{K}$ be the compositum of $K$ and $\mathrm{Frac}(T)$. If $m$ is the exponent of the Galois group of $\mathfrak{K}/\mathrm{Frac}(T)$, then by Kummer theory, choose canonical elements $g_1,\dots,g_s\in T$ such that $\mathfrak{K}=\mathrm{Frac}(T)(\omega_1,\dots,\omega_s)$ where the $\omega_i$ are $m$-th roots of the $g_i$. Let $\mathfrak{R}$ be the integral closure of $T$ in $\mathfrak{K}$. It suffices to show that $\mathfrak{R}$ admits a small Cohen-Macaulay algebra. Let $f_1,\dots,f_{r'}$ be the canonical divisors of $\mathfrak{K}/\mathrm{Frac}(T)$ in $T$ (possibly an empty list). There exist units $\alpha_1,\dots,\alpha_s\in T$ such that the integral closure of $T[\sqrt[m]{f_1},\dots,\sqrt[m]{f_{r'}},\sqrt[m]{\alpha_1},\dots,\sqrt[m]{\alpha_s}]$, say $\mathfrak{T}$, is a module finite extension of $\mathfrak{R}$. We will show that $\mathfrak{T}$ is Cohen-Macaulay, which would complete the proof. Let $\mathfrak{R}^*$ be the integral closure of $T[\sqrt[p]{f_1},\dots,\sqrt[p]{f_{r'}},\sqrt[p]{\alpha_1},\dots,\sqrt[p]{\alpha_s}]$. Note that each $\sqrt[p]{f_i}$ and each $\sqrt[p]{\alpha_i}$ is square free in $T[\sqrt[p]{f_1},\dots,\sqrt[p]{f_{r'}},\sqrt[p]{\alpha_1},\dots,\sqrt[p]{\alpha_s}]$ and hence in $\mathfrak{R}^*$. Using \Cref{HK2}, we then see that $\mathfrak{T}$ is free over $\mathfrak{R}^*$. Thus, it suffices to show that $\mathfrak{R}^*$ is Cohen-Macaulay. 
\par
Since $T_{(\psi-1)}\to \mathfrak{R}_{(\psi-1)}$ is tamely $p$-ramified, so is $T_{(\psi-1)}\to \overline{T_{(\psi-1)}[\sqrt[p]{g_i}]}$ for each $1\leq i\leq s$, where $\overline{T_{(\psi-1)}[\sqrt[p]{g_i}]}$ is the integral closure of $T_{(\psi-1)}[\sqrt[p]{g_i}]$. By \Cref{thm:p_ramification_characterization}, it follows that each $T_{(\psi-1)}\to \overline{T_{(\psi-1)}[\sqrt[p]{g_i}]}$ is \'{e}tale over $p$. Moreover, since $T_{(\psi-1)}\to \mathfrak{R}_{(\psi-1)}$ is tamely $p$-ramified and $p\mid n$, it follows that $\psi-1$ is not amongst $f_1,\dots,f_{r'}$. Now $T_{(\psi-1)}\to \overline{T_{(\psi-1)}[\sqrt[p]{g_i},\sqrt[p]{f_1},\dots,\sqrt[p]{f_{r'}}]}$ factors through $T_{(\psi-1)}\to \overline{T_{(\psi-1)}[\sqrt[p]{\alpha_i}]}$. Since $\overline{T_{(\psi-1)}[\sqrt[p]{g_i},\sqrt[p]{f_1},\dots,\sqrt[p]{f_{r'}}]}$ is integral over and birational to the join of the integral closures of $T_{(\psi-1)}[\sqrt[p]{g_i}]$, $T_{(\psi-1)}[\sqrt[p]{f_1}],\dots, T_{(\psi-1)}[\sqrt[p]{f_{r'}}]$, the former is the integral closure of the latter. For some subset $e_1,\dots,e_k$ of the list $g_i,f_1,\dots,f_{r'}$, the join of the integral closures of $T_{(\psi-1)}[\sqrt[p]{g_i}]$, $T_{(\psi-1)}[\sqrt[p]{f_1}],\dots, T_{(\psi-1)}[\sqrt[p]{f_{r'}}]$ is isomorphic to $\overline{T_{(\psi-1))}[\sqrt[p]{e_i}]}\otimes_{T_{(\psi-1)}}\dots \otimes_{T_{(\psi-1)}} \overline{T_{(\psi-1)}[\sqrt[p]{e_k}]}$ by \Cref{kernel_lemma}. Since fibre product of \'{e}tale morphisms are \'{e}tale, it follows from \Cref{thm:p_ramification_characterization} that $T_{(\psi-1)}\to \overline{T_{(\psi-1))}[\sqrt[p]{e_i}]}\otimes_{T_{(\psi-1)}}\dots \otimes_{T_{(\psi-1)}} \overline{T_{(\psi-1)}[\sqrt[p]{e_k}]}$ is \'{e}tale. Moreover, since  $\overline{T[\sqrt[p]{e_i}]}\otimes_{T}\dots \otimes_{T} \overline{T[\sqrt[p]{e_k}]}[1/p]$ is integrally closed by \Cref{HK2}, it follows that $\overline{T[\sqrt[p]{e_i}]}\otimes_{T}\dots \otimes_{T} \overline{T[\sqrt[p]{e_k}]}$ is regular in codimension one and hence integrally closed. Hence each $\alpha_i$ is $p$-unramified.
\par
Note that $S[\psi]$ is local and $S[\psi]/(\psi-1)\simeq S/pS$, so that $\psi-1$ is a minimal generator of the maximal ideal. Since $S[\psi]\to S[\epsilon]$ is an \'{e}tale map, \Cref{prop:RLR_epsilon}(1) implies that the order of $\psi-1$ with respect to each maximal ideal of $S[\epsilon]$ is $1$. Condition (2) of \Cref{def:quasi_p_unramified} implies that the order of $\psi-1$ with respect to each maximal ideal of $T$ is also $1$. In particular, the localization at each maximal ideal of $T/(\psi-1)T$ is integrally closed and hence $T/(\psi-1)T$ is integrally closed. We relabel the list $f_1,\dots,f_{r'},\alpha_1,\dots,\alpha_s$ to $f_1,\dots,f_r$. Since the image of each $f_i$ in the quotient field of $T/(\psi-1)T$ is a $p$-th power, it follows that each $f_i$ is a $p$-th power in $T/(\psi-1)T$. Write $f_i=h_i^p+(\psi-1)a_i$ for some $h_i,a_i\in T$. By \Cref{thm:p_ramification_characterization}, $\Gamma_{(\psi-1)T_{(\psi-1)}}(f_i,p)\geq p$. We have in $T_{(\psi-1)}:$

\[h_i^p+(\psi-1)a_i=(h_{1,i}/h_{2,i})^p+(b_i/c_i)(\psi-1)^p\]
for some $h_{1,i}, h_{2,i},b_i,c_i\in T$, $h_{2,i},c_i\notin (\psi-1)T$. This implies that $(h_ih_{2,i})^p-h_{1,i}^p\in (\psi-1)T$ and hence $hh_{2,i}-h_{1,i}\in (\psi-1)T$. Thus, the above equation then implies that $a_i(\psi-1)\in (\psi-1)^pT$ and thus $a_i\in (\psi-1)^{p-1}T$. In particular, each $f_i=h_i^p+(\psi-1)^pd_i$ for some $d_i\in T$. Set $\mu_i:=\sqrt[p]{f_i}$. Using \Cref{CM_dual_prop}(i) and \Cref{lem:forward_implication}(ii), we see that for $J_i:=(\mu_i-h_i,\psi-1)^{p-1}T[\mu_i]$, we have $\Hom_{T[\mu_i]}(J_i,T[\mu_i])_{(\psi-1)}\simeq \overline{T[\mu_i]}_{(\psi-1)}$. Note that by \Cref{HK2}, $T[\mu_i][1/p]$ is integrally closed. Hence by \cite[Proposition 2.1(i)]{KA}, $\Hom_{T[\mu_i]}(J_i,T[\mu_i])\simeq \overline{T[\mu_i]}$. By \Cref{CM_dual_prop}(iii), $\overline{T[\mu_i]}$ is a maximal Cohen-Macaulay $T[\mu_i]$-module and hence is a Cohen-Macaulay ring. After discarding some of the $\mu_i$ if necessary, we can assume that the $\mathrm{Frac}(T)(\mu_i)$ satisfy the linear disjointness hypothesis over $\mathrm{Frac}(T)$ in the statement of \Cref{kernel_lemma}. Thus by \Cref{kernel_lemma}, the join of the $\overline{T[\mu_i]}$ is isomorphic as a $T$-algebra to $\overline{T[\mu_1]}\otimes_T\dots \otimes_T\overline{T[\mu_r]}$. Since the former is birational to and integral over $T[\mu_1,\dots,\mu_r]$, the proof would be complete if we show that the latter is integrally closed. Since it is free over $T$ it satisfies $(S_2)$ as a $T$-module and hence as a ring. Moreover, $\overline{T[\mu_1]}\otimes_T\dots \otimes_T\overline{T[\mu_r]}[1/p]$ is integrally closed by \Cref{HK2}. Since each $T_{(\psi-1)}\to \overline{T_{(\psi-1)}[\mu_i]}$ is \'{e}tale and fibre product of \'{e}tale morphisms are \'{e}tale, $T_{(\psi-1)}\to \overline{T_{(\psi-1)}[\mu_1]}\otimes_{T_{(\psi-1)}}\dots \otimes_{T_{(\psi-1)}}\overline{T_{(\psi-1)}[\mu_r]}$ is \'{e}tale. In particular, all height one primes containing $p$ in $\overline{T[\mu_1]}\otimes_T\dots \otimes_T\overline{T[\mu_r]}$ are regular and hence it is regular in codimension one. Thus $\overline{T[\mu_1]}\otimes_T\dots \otimes_T\overline{T[\mu_r]}$ is integrally closed, so $\mathfrak{R}^*$ is Cohen-Macaulay and the proof is complete.
\end{proof}

\begin{ex}[\textbf{Koh's example}]\label{ex:Koh}

In \cite[Example 2.4]{KO}, Koh gives an example showing that the main theorem of \cite{RO} fails in the modular case. We will observe that this example is $p$-unramified, i.e., \'{e}tale in codimension one over $p$ and show that our results yield a small Cohen-Macaulay algebra over it. 

\par
Let $S,L$ be as in \Cref{thm:roberts_extension}. Assume $p=3$ and $\dim(S)\geq 3$. Let $\epsilon$ be a primitive cube root of unity. Since $\epsilon-1$ is a minimal generator of the maximal ideal of $S[\epsilon]$, it follows that $i\sqrt{3}$ is also a minimal generator of the maximal ideal. Let $i\sqrt{3},x,y$ be part of a minimal generating set of the maximal ideal of $S[\epsilon]$. Let $a:=xy^4+27$, $b:=x^4y+27$, $f:=ab^2$ and $\theta=\sqrt[3]{f}$ in some fixed algebraic closure of $L$. Let $K=L(\epsilon,\theta)$ and $R$ the integral closure of $S$ in $K$. \cite[Example 2.4]{KO} shows that $R$ is not Cohen-Macaulay. We see that $\epsilon-1$ is the unique prime divisor of $p$ in $S[\epsilon]$, $\mathrm{ord}_{\epsilon-1}(p)=2$ and that $\Gamma_{(\epsilon-1)S[\epsilon]_{(\epsilon-1)}}(f)\geq 6\geq p/(p-1)\mathrm{ord}_{\epsilon-1}(p)$. By \Cref{thm:p_ramification_characterization}, $f\in S[\epsilon]$ is $p$-unramified, i.e., $S[\epsilon]\to R$ is \'{e}tale in codimension one over $p$. The canonical divisors for $K/L(\epsilon)$ in $S[\epsilon]$ are $\{a,b\}$. Note that $\Gamma_{(\epsilon-1)S[\epsilon]_{(\epsilon-1)}}(a)=0=\Gamma_{(\epsilon-1)S[\epsilon]_{(\epsilon-1)}}(b)$, so that $K/L$ is not $p$-unramified in $S$. To see this, suppose $\Gamma_{(\epsilon-1)S[\epsilon]_{(\epsilon-1)}}(a)\geq 1$. Since $S[\epsilon]/(\epsilon-1)$ is integrally closed, it follows that $\Gamma_{(\epsilon-1)}(a)\geq 1$. Since $S[\epsilon]/(\epsilon-1)$ is regular local with the images of $x,y$ part of a minimal generating set for the maximal ideal, this is impossible. However, $K/L(\epsilon)$ is a quasi-unramified abelian extension over $S[\epsilon]$. Consider the injective map of regular local rings $f:S[\epsilon]\to T:=S[\epsilon][\sqrt[3]{x},\sqrt[3]{y}]$. Then $(\epsilon-1,\sqrt[3]{x},\sqrt[3]{y})$ is part of a minimal generating set for the maximal ideal of $T$. Then $f$ and $T$ satisfy conditions (1) and (2) of \Cref{def:quasi_p_unramified}. Since $\mathrm{Gal}(K(\sqrt[3]{x},\sqrt[3]{y})/L(\epsilon,\sqrt[3]{x},\sqrt[3]{y}))=\mathbb{Z}/3\mathbb{Z}$ and $a,b\in T$ are prime, the canonical divisors for $K(\sqrt[3]{x},\sqrt[3]{y})/L(\epsilon,\sqrt[3]{x},\sqrt[3]{y})$ in $T$ are $\{a,b\}$. Since $\Gamma_{(\epsilon-1)T_{(\epsilon-1)}}(a)\geq 6\geq (p/p-1)\mathrm{ord}_{\epsilon-1}(p)$, by \Cref{thm:p_ramification_characterization} $a\in T$ is $p$-unramified. Similarly for $b\in T$. Thus from the proof of \Cref{thm:roberts_extension}, we see that the integral closure of $T[\sqrt[3]{a},\sqrt[3]{b}]$ is Cohen-Macaulay and hence a small Cohen-Macaulay algebra for $R$. Indeed, by \Cref{lem:forward_implication}(2) and \Cref{CM_dual_prop}(1), $\overline{T[\sqrt[3]{a}]}_{(\epsilon-1)}\simeq \Hom_{T[\sqrt[3]{a}]}(J, T[\sqrt[3]{a}])_{(\epsilon-1)}$ for $J=(\sqrt[3]{a}-\sqrt[3]{x}(\sqrt[3]{y})^4, \epsilon-1)^{2}T[\sqrt[3]{a}]$. Since $T[\sqrt[3]{a}][1/p]$ is integrally closed by \Cref{HK2}, we see by \cite[Prposition 2.1(i)]{KA} that $\overline{T[\sqrt[3]{a}]}\simeq \Hom_{T[\sqrt[3]{a}]}(J, T[\sqrt[3]{a}])$. Moreover, by \Cref{CM_dual_prop}(3), $\overline{T[\sqrt[3]{a}]}$ is Cohen-Macaulay. A similar argument shows that $\overline{T[\sqrt[3]{b}]}\simeq \Hom_{T[\sqrt[3]{b}]}(I, T[\sqrt[3]{b}])$ for $I:=(\sqrt[3]{b}-\sqrt[3]{y}(\sqrt[3]{x})^4, \epsilon-1)^2T[\sqrt[3]{b}]$ and that $\overline{T[\sqrt[3]{b}]}$ is Cohen-Macaulay. Let $V$ be the join of $\overline{T[\sqrt[3]{a}]}$ and $\overline{T[\sqrt[3]{b}]}$. By \Cref{kernel_lemma}, $V\simeq \overline{T[\sqrt[3]{a}]}\otimes_T \overline{T[\sqrt[3]{b}]}$ as $T$-algebras. Since $V$ is $T$-free, it satisfies $S_2$ as a ring. Moreover, by \Cref{HK2}, $V[1/p]$ is integrally closed. By \Cref{thm:p_ramification_characterization}, $T_{(\epsilon-1)}\to \overline{T[\sqrt[3]{a}]}_{(\epsilon-1)}$ and $T_{(\epsilon-1)}\to \overline{T[\sqrt[3]{b}]}_{(\epsilon-1)}$ are \'{e}tale and since fibre product of \'{e}tale morphisms are \'{e}tale, $T_{(\epsilon-1)}\to \overline{T[\sqrt[3]{a}]}_{(\epsilon-1)}\otimes_{T_{(\epsilon-1)}} \overline{T[\sqrt[3]{b}]}_{(\epsilon-1)}$ is \'{e}tale. In particular, all height one primes containing $p$ in $V$ are regular and hence $V$ is integrally closed. Thus $V$ is the integral closure of $T[\sqrt[3]{a},\sqrt[3]{b}]$ and from what we have shown $V$ is $T$-free and hence Cohen-Macaulay. More explicitly, $V$ is $T$-free with a basis given by

\begin{align*}
    \bigg\{ & 1, \dfrac{\sqrt[3]{a}-\sqrt[3]{x}(\sqrt[3]{y})^4}{\epsilon-1}, \dfrac{(\sqrt[3]{a}-\sqrt[3]{x}(\sqrt[3]{y})^4)^2}{(\epsilon-1)^2},
    \\
    &\dfrac{\sqrt[3]{b}-\sqrt[3]{y}(\sqrt[3]{x})^4}{\epsilon-1},\dfrac{\sqrt[3]{a}-\sqrt[3]{x}(\sqrt[3]{y})^4}{\epsilon-1}\cdot \dfrac{\sqrt[3]{b}-\sqrt[3]{y}(\sqrt[3]{x})^4}{\epsilon-1},\dfrac{(\sqrt[3]{a}-\sqrt[3]{x}(\sqrt[3]{y})^4)^2}{(\epsilon-1)^2}\cdot \dfrac{\sqrt[3]{b}-\sqrt[3]{y}(\sqrt[3]{x})^4}{\epsilon-1}, 
\\
&\dfrac{(\sqrt[3]{b}-\sqrt[3]{y}(\sqrt[3]{x})^4)^2}{(\epsilon-1)^2}, \dfrac{\sqrt[3]{a}-\sqrt[3]{x}(\sqrt[3]{y})^4}{\epsilon-1}\cdot \dfrac{(\sqrt[3]{b}-\sqrt[3]{y}(\sqrt[3]{x})^4)^2}{(\epsilon-1)^2}, \dfrac{(\sqrt[3]{a}-\sqrt[3]{x}(\sqrt[3]{y})^4)^2}{(\epsilon-1)^2} \cdot \dfrac{(\sqrt[3]{b}-\sqrt[3]{y}(\sqrt[3]{x})^4)^2}{(\epsilon-1)^2} \bigg\}.
\end{align*}

\end{ex}

\section{Comments on the $p$-ramified case}

In this section we make a couple of observations concerning the $p$-ramified case. Let $S$ be a regular local ring of mixed characteristic $p>0$ with quotient field $L$. Let $K/L$ be a finite abelian extension with $p$-torsion annihilated by $p$ and $R$ the integral closure of $S$ in $K$. Let the exponent of $\mathrm{Gal}(K/L)$ be $n$ and assume $S$ possesses a primitive $n$-th root of unity. Philosophically, if $S\to R$ is tamely $p$-ramified, then the obstruction to an analog of Roberts's theorem is the existence of $p$-ramified canonical divisors for $K/L$ in $S$. This is made concrete in \Cref{cor:p-ramified}. When $S$ is complete with perfect residue field, we also exhibit a calculation in the first $p$-ramified case and show that it admits a small Cohen-Macaulay algebra of rank at most $(p-1)p^{p(d-1)+1}$ where $d=\dim(S)$.

\begin{cor}\label{cor:p-ramified}
    Let $S$ and $L$ be as in \Cref{thm:roberts_extension}. Let $K/L$ be a finite Abelian extension with $p$-torsion annihilated by $p$ and $R$ the integral closure of $S$ in $K$. Assume $S\to R$ is tamely $p$-ramified. Let $g_1,\dots,g_s$ be the canonical divisors of $K(\epsilon)/L(\epsilon)$ in $S[\epsilon]$ where $\epsilon$ is a primitive $n$-th root of unity for $n$ the exponent of $\mathrm{Gal}(K/L)$. Let $\alpha_1,\dots,\alpha_r$ be units in $S[\epsilon]$ such that each canonical element is  of the form $\alpha_i\cdot b$ for $b$ a monomial in $g_1,\dots,g_s$. Let $g_i^U$ (resp. $g_i^R$) and $\alpha_i^U$ (resp. $\alpha_i^R$) denote the $p$-unramified ($p$-ramified) elements amongst the $g_i$ and $\alpha_i$. Then the integral closure of $S[\epsilon][\sqrt[n]{g_1},\dots, \sqrt[n]{g_s},\sqrt[n]{\alpha_1},\dots, \sqrt[n]{\alpha_r}]$ admits a small Cohen Macaulay algebra (module) if and only if the integral closure of $S[\epsilon][\{\sqrt[p]{g_i^R}\}_i, \{\sqrt[p]{\alpha_i^R}\}_i\}]$ admits one.
\end{cor}

\begin{proof}
    The forward implication is obvious. Let $\mathfrak{R}^r$ denote the integral closure of $S[\epsilon][\{\sqrt[p]{g_i^R}\}_i, \{\sqrt[p]{\alpha_i^R}\}_i\}]$. Assume $\mathfrak{R}^r$ admits a small Cohen-Macaulay module. We will show that the integral closure of $S[\epsilon][\sqrt[n]{g_1},\dots, \sqrt[n]{g_s},\sqrt[n]{\alpha_1},\dots, \sqrt[n]{\alpha_r}]$ is free over $\mathfrak{R}^r$. By \Cref{lem1} the proof would then be complete. Since $S[\epsilon]_{(\epsilon^{n/p}-1)}\to \mathfrak{R}_{(\epsilon^{n/p}-1)}$ is tamely $p$-ramified and $p\mid n$, it follows that $\psi-1$ is not amongst $g_1,\dots,g_{s}$. Thus each $\sqrt[p]{g_i}$ and each $\sqrt[p]{\alpha_i}$ is square free in $T[\sqrt[p]{g_1},\dots,\sqrt[p]{g_{s}},\sqrt[p]{\alpha_1},\dots,\sqrt[p]{\alpha_r}]$ and hence square free in its integral closure $\mathfrak{R}$. Using \Cref{HK2}, we then see that the integral closure of $S[\epsilon][\sqrt[n]{g_1},\dots, \sqrt[n]{g_s},\sqrt[n]{\alpha_1},\dots, \sqrt[n]{\alpha_r}]$ is free over $\mathfrak{R}$. Hence it suffices to show that $\mathfrak{R}$ is free over $\mathfrak{R}^r$. Let $\mathfrak{R}^u$ denote the integral closure of $S[\epsilon][\{\sqrt[p]{g_i^U}\}_i, \{\sqrt[p]{\alpha_i^U}\}_i\}]$. We will show that the join $V$ of $\mathfrak{R}^U$ and $\mathfrak{R}^r$ is integrally closed, i.e., $V=\mathfrak{R}$. Discarding some elements if necessary, we may assume that the fraction fields of $\mathfrak{R}^r$ and $\mathfrak{R}^u$ are linearly disjoint over $L(\epsilon)$. From the proof of \Cref{thm:roberts_extension}, $\mathfrak{R}^u$ is $S$-free and hence a projective $S[\epsilon]$ module. Since $S[\epsilon]$ is semi-local and the rank of the localization of $\mathfrak{R}^u$ at each of its maximal ideals is constant, $\mathfrak{R}^u$ is $S[\epsilon]$-free. Hence $\mathfrak{R}^r\otimes_{S[\epsilon]}\mathfrak{R}^u$ is torsion-free. By \Cref{kernel_lemma}, $V\simeq \mathfrak{R}^r\otimes_{S[\epsilon]}\mathfrak{R}^u$. Clearly, $\mathfrak{R}^r\otimes_{S[\epsilon]}\mathfrak{R}^u$ satisfies $(S_2)$. Moreover, by \Cref{HK2}, $(\mathfrak{R}^r\otimes_{S[\epsilon]}\mathfrak{R}^u)[1/p]$ is integrally closed. Therefore $\mathfrak{R}^r\otimes_{S[\epsilon]}\mathfrak{R}^u$ is integrally closed if and only if all height one primes containing $p$ are regular. Since $S[\epsilon]_{(\epsilon-1)}\to \mathfrak{R}^u_{(\epsilon-1)}$ is \'{e}tale and a base change of an \'{e}tale morphism is \'{e}tale, we see that $\mathfrak{R}^r_{(\epsilon-1)}\to \mathfrak{R}^r_{(\epsilon-1)}\otimes_{S[\epsilon]_{(\epsilon-1)}}\mathfrak{R}^u_{(\epsilon-1)}$ is \'{e}tale. In particular, all height one primes containing $p$ in $\mathfrak{R}^r\otimes_{S[\epsilon]}\mathfrak{R}^u$ are regular and hence it is integrally closed. Since $\mathfrak{R}^u$ is $S[\epsilon]$-free it follows that $\mathfrak{R}$ is free over $\mathfrak{R}^r$ and the proof is complete.
\end{proof}

\begin{rem}
    Assume notation as in \Cref{cor:p-ramified}. If all the $g_i$ are $p$-ramified, then $S[\epsilon][\sqrt[n]{g_1},\dots,\sqrt[n]{g_s}]$ need not be Cohen-Macaulay, see \cite[Example 4.8]{PS2}.
\end{rem}

Let $S$ be a complete unramified regular local ring of mixed characteristic $p>0$ with perfect residue field $k$ and fraction field $L$. Let $\epsilon$ be a primitive $p$-th root of unity in a fixed algebraic closure of $L$. We now exhibit a construction of a small Cohen-Macaulay algebra in the case of a degree $p$ extension of $L(\epsilon)$ with the property that some canonical element in $T:=S[\epsilon]$ is square free and $p$-ramified (and hence not tamely $p$-ramified by \Cref{thm:p_ramification_characterization}). We hope that this indicates a path to understanding the $p$-ramified case in general.
\par 
Suppose $K/L(\epsilon)$ is a degree $p$ extension such that some canonical element $f\in T$ for $K/L(\epsilon)$ is square-free and $p$-ramified. Then $K=L(\epsilon,\omega)$ for $\omega=\sqrt[p]{f}$. Let $R$ be the integral closure of $T$ in $K$, i.e., the integral closure of $T[\omega]$. The result we seek is clear if $\dim(S)\leq 2$. Therefore assume $d:=\dim(S)\geq 3$. Complete $\epsilon-1$ to a minimal generating set $(\epsilon-1,x_2,\dots,x_d)$ for the maximal ideal of $T$. If $f$ is not divisible by $\epsilon-1$ or any of the $x_i$'s, set $g:=f$; if $f$ is divisible by either $\epsilon-1$ or any of the $x_i$'s set $g$ to be the quotient of $f$ by such factors. Choose a strict sequence of regular local rings as constructed in the proof of \Cref{S_p_lemma}(2) (adjoin iterated $p$-th roots of the $x_i$), $T:=\mathfrak{T}_0\subseteq \mathfrak{T}_1\subseteq \dots \subseteq \mathfrak{T}_p$ so that $g\equiv(h^p+\sum_{i=t}^{p-1}h_i^p(\epsilon-1)^i$) $mod(\epsilon-1)^p\mathfrak{T}_p$ for some $h,h_t,\dots,h_{p-1}\in \mathfrak{T}_p$ and $1\leq t\leq p-1$. Let $\alpha$ be a $p$-th root of $\epsilon-1$ in the algebraic closure of $L(\epsilon)$. Then $\mathfrak{D}:=\mathfrak{T}_p[\alpha]$ is a regular local ring with $\alpha$ a minimal generator of its maximal ideal. Then $g\in \D$ is square free by \Cref{HK1} and $g\equiv(h^p+\sum_{i=t}^{p-1}h_i^p(\alpha)^{pi}$) $mod\: (\alpha^{p^2}\mathfrak{D})$. Let $\mu$ denote a $p$-th root of $g$; it suffices to show that the integral closure of $\D[\mu]$ admits a small Cohen-Macaulay algebra. Note that $\mathrm{ord}_{\alpha}(p)=p(p-1)$. If $h_t=h_{t+1}=\dots=h_{p-1}=0$ then proceeding as in the proof of \Cref{lem:forward_implication}(1) and applying \Cref{CM_dual_prop}(3), one sees that the integral closure of $\D[\mu]$ is Cohen-Macaulay and $R$ admits a small Cohen-Macaulay algebra. Now without loss of generality assume that $h_t\neq 0$ and $h_t\notin \alpha\D$. Since $p\in \D$ is an associate of $\alpha^{p(p-1)}$, it follows that $\Gamma_{\alpha\D}(g)\geq p(p-1)+t$. In particular, we may write $g=r^p+\alpha^{p(p-1)+t}\cdot y$ for some $r,y\in \D$, $y\notin \alpha \D$. Choose an integer $1\leq l\leq p-1$ such that $lt\equiv 1 \mathrm{mod} (p\mathbb{Z})$. Consider the extension $\D\subseteq \D_l:=\D[\sqrt[l]{\alpha}]$. Then $\D_l$ is a regular local ring with $\sqrt[l]{\alpha}$ a minimal generator of the maximal ideal. Moreover, by \Cref{HK1}, $g\in \D_l$ is square free. Denote the unique height one prime ideal containing $p$ in $\D_l[\mu]$ by $P_l:=(\sqrt[l]{\alpha},\mu-r)$. We have in $A:=\D_l[\mu]$ 

\begin{align}\label{eq1_Dw}
 (\sqrt[l]{\alpha})^{lp(p-1)+tl}\cdot y &=\mu^p -r^p \nonumber
 \\
 &= (\mu-r)^p+p\cdot c'\cdot (\mu-r)
 \end{align}
where $c'\in A$ is the image in $A$ of the element $C'\in \D_l[W]$ in \Cref{P6}. Recall that $p=-(c'_{\epsilon})^{-1}(\epsilon-1)^{p-1}$ where $c'_{\epsilon}$ is the image in $\mathbb{Z}[\epsilon]$ of the corresponding element $C'_{\epsilon}\in  \mathbb{Z}[W]$ from \Cref{P6}. We have

\begin{equation}
    (\mu-r)^p-((c'_{\epsilon})^{-1}\cdot c')(\mu-r)(\sqrt[l]{\alpha})^{lp(p-1)}-(\sqrt[l]{\alpha})^{lp(p-1)+tl}y=0.
\end{equation}

Write $lp(p-1)+tl=pq+1$ for an integer $q$. Note that $lp>q$. Dividing the above equation by $(\sqrt[l]{\alpha})^{pq}$ and setting $\zeta:=(\sqrt[l]{\alpha})^{-q}(\mu-r)$, we see that $\zeta$ is a root of the polynomial in $A[X]$

\begin{equation}\label{zeta_integral_equation}
    X^p-((c'_{\epsilon})^{-1}\cdot c')(\sqrt[l]{\alpha})^{(p-1)(lp-q)}X-\sqrt[l]{\alpha}\cdot y.
\end{equation}

Since $\mathfrak{D}_l$ is universally catenary, it follows that $A[\zeta]$ has a unique height one prime containing $p$ and is generated locally by $(\sqrt[l]{\alpha},\mu-r,\zeta)$. However, $\zeta\cdot (\sqrt[l]{\alpha})^q=\mu-r$ and \Cref{zeta_integral_equation} shows that $\sqrt[l]{\alpha}$ is a multiple of $\zeta$ locally. So, the unique height one prime containing $p$ in $A[\zeta]$ is regular and by \Cref{HK1}, $A[\zeta]$ is regular in codimension one. Note that $(p-1)q\leq pq+1\leq \Gamma_{(\sqrt[l]{\alpha})}(g)$ and $(p-1)q<(p-1)pl=\mathrm{ord}_{\sqrt[l]{\alpha}}(p)$. Therefore, setting $J:=(\mu-r,\alpha^{q/l})A$, we see from \Cref{CM_dual_prop}(1), that $J^{-1}_{A}=A[\zeta]$. Thus $A[\zeta]$ satisfies $(S_2)$ and is hence integrally closed. In particular, $A[\zeta]$ is a module finite algebra extension of $R$. Moreover, by \Cref{CM_dual_prop}(3), $A[\zeta]$ is Cohen-Macaulay so that it is a small Cohen Macaulay algebra for $R$. Finally, it follows from construction that $R$ admits a small Cohen-Macaulay module of rank at most $(p-1)p^{p(d-1)+1}$.


 \end{document}